\documentclass{amsart}

	\expandafter\let\csname ver@amsthm.sty\endcsname\relax
	\let\theoremstyle\relax

	\let\qedhere\relax
	
	\let\openbox\relax

\usepackage{hyperref}
\usepackage{amsthm} % Part of the above-mentioned hack
\usepackage[capitalise]{cleveref}
\usepackage{amssymb}
\usepackage{mathabx}
\usepackage{mathtools}
\usepackage{bbm}

\usepackage[T1]{fontenc} %Allows diacritics to be copied from the pdf properly
\usepackage[utf8]{inputenc} %Allows diacritics to be included in the source

\usepackage{color,soul}
\definecolor{ItalianApricot}{rgb}{1,0.7,0.5}
\sethlcolor{ItalianApricot}

\theoremstyle{plain}
\newtheorem{thm}{Theorem}[section]
\newtheorem{theorem}[thm]{Theorem}
\newtheorem{prop}[thm]{Proposition}
\newtheorem{proposition}[thm]{Proposition}
\newtheorem{lem}[thm]{Lemma}
\newtheorem{lemma}[thm]{Lemma}
\newtheorem{cor}[thm]{Corollary}

\newtheorem{fact}[thm]{Fact}
\theoremstyle{definition}

\newtheorem{definition}[thm]{Definition}

\newtheorem{question}[thm]{Question}
\theoremstyle{remark}
\newtheorem{remark}[thm]{Remark}
\newtheorem{claim}{Claim}[thm]
\newtheorem{conjecture}[thm]{Conjecture}
\numberwithin{equation}{section}

\theoremstyle{plain}
\newcommand{\thistheoremname}{}
\newtheorem*{genericthm}{\thistheoremname}
\newenvironment{namedthm}[1]
	{\renewcommand{\thistheoremname}{#1}%
	\begin{genericthm}}
 	{\end{genericthm}}

\newcommand{\NumberQED}[1]{
\renewcommand\qedsymbol{\ensuremath{\openbox}$_{\ref{#1}}$}	
}

\renewcommand{\epsilon}{\varepsilon}
\renewcommand{\phi}{\varphi}
\newcommand{\seq}[1]{{\left\langle{#1}\right\rangle}}
\newcommand\+[1]{\mathcal{#1}}

\DeclareMathOperator{\dom}{dom}

\newcommand{\floor}[1]{\left\lfloor{#1}\right\rfloor}

\newcommand{\tth}{{}^{\textup{th}}}
\newcommand{\conc}{\hat{\,\,}}
\newcommand{\andd}{\,\,\,\&\,\,\,}

\newcommand{\converge}{\!\!\downarrow}

\renewcommand{\setminus}{\smallsetminus}
\newcommand{\w}{\omega}
\newcommand{\s}{\sigma}

\renewcommand{\le}{\leqslant}
\renewcommand{\ge}{\geqslant}
\renewcommand{\leq}{\leqslant}
\renewcommand{\geq}{\geqslant}
\renewcommand{\preceq}{\preccurlyeq}
\renewcommand{\succeq}{\succcurlyeq}
\newcommand{\nle}{\nleqslant}

\newcommand{\Tur}{\textup{\scriptsize T}}
\newcommand{\leT}{\le_{\Tur}}
\newcommand{\geT}{\ge_{\Tur}}

\newcommand{\limcost}{\underline{\cost}}

\newcommand{\then}{\rightarrow}

\newcommand{\Nat}{\mathbb{N}}

\newcommand \bool[1]{\left\ldbrack #1 \right\rdbrack}

\newcommand{\leb}{\mu}
\DeclareMathOperator{\upto}{\upharpoonright}

\newcommand{\uhr}[1]{\upto{#1}}
\newcommand{\rest}[1]{\upto{#1}}

\newcommand{\cost}{\mathbf{c}}
\newcommand{\dost}{\mathbf{d}}

\newcommand{\DII}{\Delta^0_2}
\newcommand{\NN}{{\mathbb{N}}}
\newcommand{\RR}{{\mathbb{R}}}
\newcommand{\QQ}{{\mathbb{Q}}}

\newcommand{\sub}{\subseteq}

\newcommand{\ML}{\textup{\scriptsize{ML}}}
\newcommand{\SI}[1]{\Sigma^0_{#1}}

\newcommand{\Halt}{{\ES'}}
\newcommand{\ES}{\emptyset}

\newcommand{\ria}{\rightarrow}
\newcommand{\tp}[1]{2^{#1}}

\newcommand{\fa}{\forall}

\newcommand{\lwtt}{\le_{\mathrm{wtt}}}

\newcommand{\sss}{\sigma}

\newcommand{\wt}{\widetilde}

\newcommand{\lra}{\leftrightarrow}
\newcommand{\LR}{\Leftrightarrow}
\newcommand{\RA}{\Rightarrow}
\newcommand{\LA}{\Leftarrow}

\definecolor{lightred}{rgb}{1,.60,.60}

\begin{document}

% \title[Strengthening $K$-triviality via random oracles]{Strengthening \boldmath $K$-triviality via random oracles}

%\title{Interaction of $K$-trivial sets with random oracles via cost functions}

 \title{Martin-L\"of reducibility and cost functions}

\date{\today}
%\renewcommand{\datename}{Last compilation:}
%\date{\today.\\\indent Last time the following date was changed: 4 June 2017}

\author{Noam Greenberg}
\address{School of Mathematics and Statistics\\
Victoria University of Wellington\\
Wellington, New Zealand}
\email{greenberg@msor.vuw.az.nz}

\author{Joseph S.~Miller}
\address{Department of Mathematics\\
University of Wisconsin\\
Madison, WI 53706, USA}
\email{jmiller@math.wisc.edu}

\author{Andr\'e Nies}
\address{Department of Computer Science\\
University of Auckland\\
Private Bag 92019\\
Auckland, New Zealand}
\email{andre@cs.auckland.ac.nz}

\author{Dan Turetsky}
\address{School of Mathematics and Statistics\\
Victoria University of Wellington\\
Wellington, New Zealand}
\email{dan@msor.vuw.ac.nz}
 
\thanks{Greenberg and Nies were supported by the Marsden Fund of New Zealand. Greenberg was also supported by a Rutherford Discovery Fellowship from the Royal Society of New Zealand. This research was started during a retreat at the Research Centre Coromandel.}


% 03-XX: Mathematical logic and foundations
%	03Dxx: Computability and recursion theory
%		03D25: Recursively (computably) enumerable sets and degrees
%		03D28: Other Turing degree structures
%		03D30: Other degrees and reducibilities
%		03D32: Algorithmic randomness and dimension [See also 68Q30]
% 
% 68-XX: Computer science For papers involving machine computations and programs in a specific mathematical area, see Section--04 in that area
%	68Qxx: Theory of computing
%		68Q30: Algorithmic information theory (Kolmogorov complexity, etc.) [See also 03D32]

\begin{abstract} 
Martin-L\"of (ML)-reducibility compares the complexity of $K$-trivial sets of natural numbers by examining the Martin-L\"of random sequences that compute them. One says that a $K$-trivial set $A$ is ML-reducible to a $K$-trivial set $B$ if every ML-random computing $B$ also computes $A$. We show that every $K$-trivial set is computable from a c.e.\ set of the same ML-degree. We investigate the interplay between ML-reducibility and cost functions, which are used to both measure the number of changes in a computable approximation, and the type of null sets intended to capture ML-random sequences. We show that for every cost function there is a c.e.\ set that is ML-complete among the sets obeying it. We characterise the $K$-trivial sets computable from a fragment of the left-c.e.\ random real~$\Omega$ given by a computable set of bit positions. This leads to a new characterisation of strong jump-traceability.
\end{abstract}

\maketitle
\tableofcontents

%%%%%%%%
%%%%%%%%
\section{Introduction}
%%%%%%%%
Martin-L\"of (ML) randomness and $K$-triviality are antipodal properties of sets of natural numbers. Nonetheless, sets of the two kinds interact in interesting ways via Turing reducibility. For instance, combining the results of~\cite{Bienvenu.Greenberg.ea:16,DayMiller:Covering}, a c.e.\ set is $K$-trivial if and only if it is computable from a Turing incomplete ML-random set; see~\cite{CoveringProblemAnnouncement}. 

Our purpose is to study the relative complexity of the $K$-trivial sets via a preordering coarser than Turing reducibility %$\leT$
that is given by this interaction: $A \le_\ML B$ if every ML-random computing $B$ also computes $A$. This preordering, called ML-reducibility, was introduced in~\cite{Bienvenu.Greenberg.ea:16} by Bienvenu, Ku\v cera, and three of the authors of the present paper. They showed that there is an ML-complete $K$-trivial (which they called ``smart''). While the $K$-trivials appear somewhat amorphous under Turing reducibility, we will show that an interesting structure emerges when they are viewed through the lens of this preordering. Research in this direction was carried out first by three of the authors of the present paper in~\cite{SubclassesPaper}. They described a dense linear hierarchy of natural principal ideals in the ML-degrees of the $K$-trivials. Our first result, Theorem~\ref{thm:ML_degrees_ce_generated}, shows that each $K$-trivial is ML-equivalent to a c.e.\ $K$-trivial, so the structure we find on the $K$-trivials is fully given by c.e.\ witnesses. 

\smallskip

\noindent\emph{Some background on ML-randomness and $K$-triviality.} Sets of natural numbers (often simply referred to as \emph{sets}) will be identified with infinite bit sequences. ML-randomness is central among the notions of randomness given by algorithmic tests. A set $Z\subseteq \omega$ is ML-random (sometimes just called ``random'' in this paper) if $Z \not \in \bigcap_m G_m$ for any sequence of uniformly $\Sigma_1^0$ sets in Cantor space such that the Lebesgue measure of $G_m$ is at most $2^{-m}$. A sequence $\seq{G_m}$ of this kind is called an ML-test.

Chaitin's $\Omega$, the halting probability of a universal prefix-free machine, is an example of an ML-random sequence. Note that $\Omega$ is Turing equivalent to the halting problem. There are various ways to build a low ML-random sequence. One way is to use the low basis theorem together with the fact that there is a universal ML-test. Another is to take $\Omega_R$, the bits of $\Omega$ with location in an infinite, co-infinite computable set $R$; to see that $\Omega_R$ is low one can e.g.\ combine \cite[Prop.\ 3.4.10]{Nies:book}, due to \cite{Nies.Stephan.ea:05},  with van Lambalgen's Theorem~\cite{vLamb:90}. 

Let $K(x) $ denote the prefix-free descriptive complexity of a string $x$. One says that a set $A$ %(read: a subset of $\omega$)
is $K$-trivial if there is a constant $b$ such that $K(A\uhr n) \le K(n) + b$ for each $n$. (Here the ``$n$'' in $K(n)$ is interpreted as a string, for instance the string obtained by writing~$n$ in binary.) Note that $K(n)$ is, up to a constant not depending on $n$, the lowest complexity possible for a string of length $n$. So $K$-trivial sets have minimal initial segment complexity, again up to a constant. The Levin--Schnorr theorem states that a sequence $Z$ is ML-random iff $K(Z\uhr n) \ge n-d$ for some $d$ only depending on~$Z$. As $K(n) \le 2\log n+O(1)$, the notion of $K$-triviality is indeed antipodal to ML-randomness: $K$-trivial by definition means far from random. 

Eighteen or so characterisations of the $K$-trivials are known presently, most of them saying that the set is in some sense close to computable. For instance, $A$ is $K$-trivial iff $A$ is low for ML-randomness, in the sense that each ML-random is ML-random relative to $A$~\cite{Nies:AM}. For more recent ones, $A$ is $K$-trivial iff for each ML-random $Y$, the symmetric difference $Y \triangle A$ is ML-random~\cite{KuyperMiller}; $A$ is $K$-trivial iff for all $Y$ such that $\Omega$ is $Y$-random, $\Omega$ is $Y \oplus A$-random~\cite{TwoMoreCharacterisations}.

Despite these characterisations, and the detailed knowledge of the class of $K$-trivials they appeared to convey, paradoxically, not much progress had been made on the internal structure of the class since the early papers, such as \cite{HirschfeldtNiesStephan:UsingRandomSetsAsOracles,Nies:AM}. It was known that the $K$-trivials are downward closed under Turing reducibility, and that they determine an ideal in the Turing degrees that is contained in the superlow degrees, generated by its c.e.\ members, and has no greatest degree (i.e., it is nonprincipal). In hindsight it appears that Turing reducibility was the wrong preordering to analyse the internal structure. We will use the {coarser} ML-reducibility to amend this lack of information. %\andre{Should we mention SJT here as an example of internal  structure that was known early on?}
\smallskip
%
%\noindent \emph{Computed by a random.} 

We briefly review some results that led to the formulation of ML-reducibility. The Ku\v cera--Gacs Theorem \cite{Gacs:EverySequence,Kucera:85} states that every set $A$ is Turing below some ML-random set $Z$. If $A$ is $\Delta^0_2$ then one can take $\Omega$ as $Z$. 
What can a Turing incomplete ML-random set compute? Ku\v cera~\cite{Kucera:86} showed that   each $\Delta^0_2$ ML-random~$Z$ is Turing above a noncomputable c.e.\ set. Hirschfeldt, Nies, and Stephan~\cite{HirschfeldtNiesStephan:UsingRandomSetsAsOracles} proved that if $Z$ is Turing  incomplete, then $Z$ is necessarily ML-random relative to any such c.e.\ set $A$ it computes. By definition, this means that $A$ is a \emph{basis for ML-randomness}, which implies that it is $K$-trivial~\cite{HirschfeldtNiesStephan:UsingRandomSetsAsOracles}. Since being $K$-trivial also means close to computable~\cite{Nies:AM}, this shows that an incomplete ML-random can only compute c.e.\ sets that are close to computable. In the other direction, the works~\cite{CoveringProblemAnnouncement,Bienvenu.Greenberg.ea:16,DayMiller:Covering} mentioned in the first paragraph show that there is, in fact, a {single} incomplete $\DII$ ML-random Turing above all the $K$-trivials.

 \smallskip
 
% \noindent \emph{Computational feebleness can help.} 
% 
%%

 \smallskip
 
 \noindent \emph{Complexity classes of $K$-trivials.} 
 The $K$-trivials are closed downward under $\le_\ML$ by Day and Miller~\cite{DayMiller:Covering}. By an \emph{ML-complexity class} in the $K$-trivials we mean a subclass that determines an ML-ideal, namely, it is closed downward under $\le_\ML$, and closed under the join operator $\oplus$. Several works from 2012 on can be viewed as studies of such classes. For example, a set is strongly jump-traceable (see the beginning of Section~\ref{sec:SJT} for the  definition) if and only if it is computable from all ML-random sequences that are $\omega$-c.a., that is, weak truth table below the halting problem~\cite{GreenbergHirschfeldtNies:2012}. This implies that the strongly jump-traceables form an ML-ideal.
 
 Given an ML-complexity class $\mathcal C$ of $K$-trivials, one can ask the following two general questions:
 
\begin{enumerate} 
\item[(a)] Can one describe the class only referring to properties of its members, rather than to randoms that compute them? 
\item[(b)] Is some set ML-complete for the class? In other words, is the ML-ideal principal?
\end{enumerate}

As an example of a description in (a) consider the original definition~\cite{FigueiraNiesStephan:SJT} of strong jump traceability of a set $A$, which (as the name indicates) refers to a way to tightly approximate the values of $J^A$, the Turing jump function of $A$.

A general way to formulate a condition of lowness among the $K$-trivials is by restricting the changes of computable approximations. Recall that each $K$-trivial~$A$ is $\Delta^0_2$ and hence has a computable approximation $\seq{A_s}$. A \emph{cost function} is a computable function \begin{center} $ \cost\colon \NN \times \NN \ria \{r \in \RR \,:\, r \ge 0\}$, \end{center} which is typically chosen to be nondecreasing in $s$, nonincreasing in $x$, and to satisfy the \emph{limit condition}, namely,   the asymptotic cost $\cost(x)=\sup_s \cost(x,s)$ approaches $0$ as $x$ increases.
% Joe: Not the limit condition?
% and satisfies the limit condition that the asymptotic cost $\cost(x)=\sup_s \cost(x,s)$ is finite for each $x$.
The idea is that at stage $s$, the least $x$ such that $A_s(x)$ changes incurs the cost $\cost(x,s)$, which subsumes the cost of changes at larger numbers at the same stage; a set $A$ \emph{obeys} a cost function~$\cost$ if it has a computable approximation that is sufficiently ``inert'' in that it incurs a finite total cost. (This means more than that there are few changes; it also means that changes need to be carried out in ``blocks'', which saves costs because only the change at the least number is counted.) It is a basic fact that each cost function that satisfies the limit condition is obeyed by a noncomputable c.e.\ set. Building on previous results, Nies~\cite[Thm.\ 4.3]{Nies:CalculusOfCostFunctions} showed that obedience to the cost function $\cost_\Omega(x,s)= \Omega_s-\Omega_x$ characterizes the $K$-trivials. 

We will show that an affirmative answer to (a) via obedience to a cost function implies an affirmative answer to (b). Given a cost function $\cost$ that is at least as strong as $\cost_\Omega$, we show as an easy corollary to our second main result, Theorem~\ref{thm:smarts_exist}, that {\it some c.e.\ set $A$ obeys $\cost$ and is ML-complete among the sets obeying $\cost$.} So if a cost function $\cost$ describes a complexity class, then that class has an ML-complete member. In particular, this holds for the class of all $K$-trivials, which therefore determines a principal ML-ideal, a result that was first obtained in~\cite{Bienvenu.Greenberg.ea:16} using similar methods. (At present this appears to be essentially the only known way to show a complexity class in the $K$-trivials has an ML-complete member.) Given a low c.e.\ set $B$, some c.e.\ set $A \not \leT B$ obeys $\cost$ by~\cite[Thm.\ 5.3.22]{Nies:book}. So the sets obeying $\cost$ (being $K$-trivial, and hence low) never form a Turing principal ideal. This lends support to our thesis that, compared to Turing reducibility,  the coarser ML-reducibility leads to a more satisfying complexity theory of the $K$-trivials. 

The class of half-bases is a further example of a complexity class that can be described by a cost function. It yields one level of the dense hierarchy of ML-ideals described in~\cite{SubclassesPaper} which we alluded to earlier on. One says that a set $A$ is a \emph{half-base} if there is an ML-random $Y$ such that $A\leT Y_0,Y_1$, where $Y_0$ consists of the bits in the even, and $Y_1$ of the bits in the odd positions. Note that each half-base is a basis for ML-randomness by van Lambalgen's theorem, and hence $K$-trivial by~\cite{HirschfeldtNiesStephan:UsingRandomSetsAsOracles}. By~\cite[Thm.\ 1.1.]{SubclassesPaper}, one can require that $Y= \Omega$; furthermore, by~\cite[Thm.\ 1.3.]{SubclassesPaper} the class of half-bases can be described by the cost function $\cost_{\Omega, 1/2}(x,s)= \sqrt{\cost_\Omega(x,s)}$. Larger cost functions are harder to obey, in a sense made precise in~\cite[Thm.\ 3.4]{Nies:CalculusOfCostFunctions}. So $\cost_{\Omega, 1/2}$ describes a proper subclass of the $K$-trivials. The properness of the  inclusion of the class of half bases in the $K$-trivials   is a result first obtained in~\cite[Thm.\ 1.3]{Bienvenu.Greenberg.ea:16}. By our method, there is an ML-complete half-base. 

%
% ML-reducibility is defined by quantifying over all ML-random sets, the only way to prove the existence of an ML-complete $A$ for a cost function $\cost$ appears to be ensuring that the ML-randoms that compute $A$ are restricted.
% 
 
 By definition, a set $A$ is ML-complete for a class $\mathcal C$ of $K$-trivials if it has the least class of ML-randoms computing it among the members of~$\mathcal C$. 
 We build a c.e.\ set $A$ that is ML-complete for the sets obeying $\cost$ by showing that ML-randoms computing $A$ cannot be very random, in the sense that they fail a generalised  type of ML test called a \emph{$\cost$-test}, where the convergence of the measure of $G_m$ to $0$ is bounded by $O(\cost(m))$, rather than $2^{-m}$ (recall here that $\cost(m)= \sup_s \cost(m,s)$). We then use the important, if easy Proposition~\ref{prop:basic fact} below: any set $B$ obeying $\cost$ is Turing below any ML-random failing such a test. 
 
 Our Theorem~\ref{thm:smarts_exist} extends the basic fact that every cost function $\cost$ is obeyed by a noncomputable c.e.\ set. The c.e.\ set $A$ we build obeys $\cost$, but ``only just'', in the sense that the collection of ML-randoms computing $A$ is as small as possible. We call such a set $A$ \emph{smart for $\cost$}, continuing the terminology of \cite{Bienvenu.Greenberg.ea:16} for $\cost_\Omega$; the theorem states that each cost function has a smart set. Then we verify, as an easy consequence of this existence of a smart set, that smartness for $\cost$ coincides with ML-completeness for $\cost$.

 The previous works~\cite{Bienvenu.Greenberg.ea:16,SubclassesPaper} and, in particular, the present paper show that far from being an obstacle, the fact that $K$-trivials are close to computable can be advantageous for the study of their relative computational complexity. Tools can be applied that would not work for computationally more complex~$\Delta^0_2$ sets. The main idea described above is to differentiate between $K$-trivials via the incomplete randoms that compute them. In contrast, a c.e.\ set that is not $K$-trivial only allows randoms above the halting problem to compute it by~\cite{HirschfeldtNiesStephan:UsingRandomSetsAsOracles}; such c.e.\ sets all have the same ML-degree.
 
\smallskip
\noindent \emph{The $K$-trivials Turing below a fragment $\Omega_R$ of $\Omega$.} Bit sequences derived in some way from Chaitin's $\Omega$ often play a special role in the algorithmic theory of randomness. For an infinite computable set $R$, we have defined above the sequences $\Omega_R$ of bits of $\Omega$ with a location in $R$. \cref{sec:comparability_of_fragments} introduces a cost function $\cost_{\Omega,R}$ that describes the class of $K$-trivials computable from $\Omega_R$. As a main result of this paper, we show in \cref{thm:criterion} that this cost function essentially only depends on the function $n \to |R \cap n|$ that gives the number of elements of $R$ less than $n$, taken up to an additive constant. For instance, if $R$ is the set of even numbers and $S$ the odd numbers, then the cost functions corresponding to $R$ and $S$ are equivalent as far as obedience goes, and hence the $K$-trivials below $\Omega_R$ coincide with the $K$-trivials below $\Omega_S$. 
This can be extended to $k/n$ bases, for $1\le k< n$, in the sense of~\cite{SubclassesPaper}, where one takes as $R$ any union of $k$ sets of the form $n \NN + r$, $0 \le r < n$. Let $B_{k/n}$ be a smart set for the corresponding cost function. Via the results in~\cite{SubclassesPaper} these sets determine a chain in the ML-degrees of $K$-trivials that is isomorphic to $(0,1)_\QQ$. For detail see the discussion around~\cref{thm:k_n-bases}. 

As an application of \cref{thm:criterion}, in \cref{thm:Omega_R_is_feeble} we show that the $K$-trivials computable from $\Omega_{R}$ are exactly those that obey $\cost_{\Omega, R}$, as promised. %Thus the class of $K$-trivials below $\Omega_R$ is the least possible.
Given a cost function $\cost$, an ML-random set $Y$ failing a $\cost$-test will be called \emph{feeble for $\cost$} if the only $K$-trivials it computes are the ones that obey $\cost$. This notion is dual to smartness for  $K$-trivials. In this language, we show that $\Omega_R$ is feeble for $\cost_{\Omega,R}$. 

%More generally, we show that for each $p \in (0,1) \cap \mathbb Q$, the level $p$ of the hierarchy is a principal ML-ideal that is described by the cost function $\cost_{\Omega, p}(x,s)= \cost_\Omega(x,s)^p$.

\smallskip
\noindent \emph{Structure of the $\le_\ML$-degrees of $K$-trivials.} By its definition, ML-reducibility is a weakening of Turing reducibility. The least degree consists of the computable sets. The usual join operation $\oplus$ induces a least upper bound in the ML-degrees.

Dual to Theorem~\ref{thm:smarts_exist}, we will show in Proposition~\ref{prop:A_is_smart_for_c_A} that each $K$-trivial $A$ is smart for some cost function $\cost_{(A)}$ that can be obtained uniformly from $A$. This yields further degree theoretical information in~\cref{rem: ML structure}. Firstly,  there are no minimal pairs in the ML-degrees of $K$-trivials. Secondly,   no ML-degree of a noncomputable $K$-trivial contains a maximal  Turing degree; in particular, it  contains an infinitely ascending chain of  Turing degrees. We do not know whether   the   ML-degree of a noncomputable $K$-trivial can contain a minimal Turing degree. 

% The smart $k/n$ bases $B_{k,n}$ yield a dense chain in the ML-degrees. 
As a further, more powerful application of \cref{thm:criterion} we will show in \cref{thm:infinite_antichain} that every countable partial ordering is embeddable into the ML-degrees of $K$-trivial sets. Also, based on a method of Ku\v cera~\cite{Kucera:86}, we obtain a pair of incomparable degrees below each non-zero $K$-trivial ML-degree. In fact, we show that for each noncomputable c.e.\ $K$-trivial $D$,  there are c.e.\ sets $A,B\leT D$ such that $A \mid_\ML B$. 

Only basic facts are known  on ML-reducibility outside the $K$-trivials. Since each Turing degree above $\Halt$ contains a ML-random, for sets above $\Halt$ the two reducibilities coincide. More generally, for PA-complete sets $A,B$, we have \begin{center} $A \le_{ML} B \LR A\oplus \Halt  \le_T B \oplus \Halt$. \end{center}  (We thank  the referee for pointing this out.) For the direction from left  to right, one uses the result of Stephan~\cite{Stephan:06} that the only PA-complete ML-randoms are the ones above $\Halt$.

In the final section, Section~\ref{sec:SJT}, we connect obedience of the cost functions $\cost_{\Omega,R}$ to strong jump traceability, the aforementioned very strong  lowness notion. While the original definition~\cite{FigueiraNiesStephan:SJT} was combinatorial, Hirschfeldt et al.\ \cite{GreenbergHirschfeldtNies:2012} showed that a set is strongly jump traceable iff it is Turing below each $\omega$-c.a.\ ML-random $Y$. We prove that $A$ is strongly jump traceable iff it obeys all the cost functions $\cost_{\Omega,R}$ for infinite computable $R$. In particular, it suffices to let the $\omega$-c.a.\ sets $Y$ as above range over the fragments~$\Omega_R$.

\begin{remark} \label{rem: badly behaved} The correspondence between cost functions and ML-degrees is incomplete. Firstly, a~cost function~$\cost$ determines an ML-degree, that of the sets which are smart (equivalently, ML-complete) for~$\cost$. However, not every set in that degree obeys~$\cost$. Secondly, every $K$-trivial set~$A$ is smart for some cost function~$\cost_{(A)}$. However, this cost function is not determined by the ML-degree of~$A$; in fact, in \cref{thm:shift} we construct an example of a set~$A$ such that even the set that results from $A$ by removing the first bit does not obey $\cost_{(A)}$. \end{remark}

Ideally, we could characterise ML-reducibility on $K$-trivials in terms of which cost functions they obey. This would give a satisfying positive answer to the following question, which remains open:

\begin{question}\label{qu:ML arithm}
Is the relation $\le_\ML$ on the $K$-trivial sets arithmetical?
\end{question}

In fact, a weaker question remains open: whether ML-completeness among the $K$-trivals is arithmetical. Another question that remains unsettled is whether the ML-degrees of $K$-trivials are dense. Given that Question~\ref{qu:ML arithm} remains open, it is hard to envisage a requirement-based construction showing density, as this would need some sort of effective listing of ``ML-reduction procedures''. Cost function-based methods appear to be insufficient here.

%%%%%%%%%%%%%%%%
%%%%%%%%
\section{Some formal definitions and facts} \label{s:costf}
In this section, for easy reference, we provide formal definitions of some of the notions discussed above. We discuss some technical detail and basic connections that will be important for the rest of the paper.

\begin{definition}[\cite{Bienvenu.Greenberg.ea:16}] \label{def:ML-reducibility}
For sets~$A$ and~$B$, we write $A \le_\ML B$ if $B \leT Y$ implies $A \leT Y$ for every ML-random sequence~$Y$.
\end{definition}

%
% In \cite{Bienvenu.Greenberg.ea:16}, it is shown that there is a greatest $K$-trivial ML-degree; on the other hand, it is well-known that there is no greatest Turing degree of $K$-trivial sets \cite[5.3.22]{Nies:book}; also see the end of our Section~\ref{sec:smart}.

\smallskip
Cost functions were introduced in \cite[Section~5.3]{Nies:book} and developed further in \cite{GreenbergNies:benign,Nies:CalculusOfCostFunctions}. 

\begin{definition} \label{def:cost_function}
A \emph{cost function} is a computable function
\[
\cost\colon \NN \times \NN \ria \{r \in \RR \,:\, r \ge 0\}.
\] 
\end{definition}

We only consider \emph{monotonic} cost functions (satisfying $\cost(x,s)\le \cost(x,s+1)$ and $\cost(x,s)\ge \cost(x+1,s)$) that have the \emph{limit condition}: for all~$x$, $\limcost(x) = \lim_s \cost(x,s)$ exists, and $\lim_{x} \limcost(x)= 0$. Further, we assume that $\cost(x,s)=0$ when $x\ge s$. 

The original purpose of cost functions was to quantify the number of changes required in a computable approximation of a $\Delta^0_2$ set~$A$: $\cost(x,s)$ is the cost of changing at stage~$s$ our guess about the value of $A(x)$. Monotonicity means that the cost of a change increases with time, and that changing the value at a smaller number is more costly. Formally:

\begin{definition}[\cite{Nies:book}] \label{def:obeying_a_cost_function}
Let $\seq{A_s}$ be a computable approximation of a $\Delta^0_2$ set~$A$, and let~$\cost$ be a cost function. The \emph{total $\cost$-cost} of the approximation is 
\[ \cost \seq{A_s} = \sum_{s\in\omega} \left\{ \cost(x,s) \,: \, x \text{ is least such that } A_{s-1}(x) \ne A_{s}(x) \right\}.\]
We say that a $\Delta^0_2$ set~$A$ \emph{obeys}~$\cost$ if the total $\cost$-cost of \emph{some} computable approximation of~$A$ is finite. We write $A \models \cost$. For cost functions $\cost$ and $\dost$, we write $\cost \to \dost$ if $A \models \cost$ implies $A\models \dost$ for each $\DII$ set $A$. By \cite[Thm.\ 3.4]{Nies:CalculusOfCostFunctions}, this is equivalent to $\underline \dost \le^\times \limcost$. 
\end{definition}
The basic existence theorem for cost functions, e.g., described in \cite[Thm.~2.7(i)]{Nies:CalculusOfCostFunctions}, says that if a cost function $\cost$ has the limit condition, then some non-computable c.e.\ set obeys $A$.
As mentioned, an important example of a cost function is $\cost_\Omega(x,s) = \Omega_s-\Omega_x$, where $\seq{\Omega_s}$ is an increasing sequence of rational numbers converging to a left-c.e., ML-random real~$\Omega$. A set obeys this cost function if and only if it is $K$-trivial (\cite[Thm.~4.3]{Nies:CalculusOfCostFunctions}, which modified a result for a related cost function in \cite{Nies:AM}). 

\begin{definition} \label{def:MLComp}
Let $\cost$ be a cost function and let $A$ be a $\DII$ set. We say that $A$ is \emph{ML-complete for $\cost$} if $A \models \cost$, and $\fa B \, [B \models \cost \RA B \le_\ML A] $.
\end{definition}
Note that the implication arrow only goes from left to right; it is not true in general that the class of sets obeying a cost function is well behaved (see Remark~\ref{rem: badly behaved}).

\smallskip
We next add some formal detail to our discussion of Theorem~\ref{thm:smarts_exist} above, that each cost function has a smart c.e.\ set. As mentioned, cost functions can also be used to introduce randomness notions between weak 2-randomness and ML-randomness.

\begin{definition}[\cite{Bienvenu.Greenberg.ea:16}, Def.\ 2.13] \label{def:cost-bounded_test} 	
Let~$\cost$ be a cost function. A sequence $\seq{V_n}$ of uniformly c.e.\ open sets such that $V_n \supseteq V_{n+1}$ for each $n$ is a \emph{$\cost$-bounded test} (or $\cost$-test for short) if $\leb(V_n) \le^\times \limcost(n)$ for all~$n$. \end{definition}

We say that such a test \emph{captures} a set $Y$ if $Y \in \bigcap_n V_n$. We also say that such a~$Y$ \emph{fails} the test. A sequence is \emph{$\cost$-random} if it fails no $\cost$-test. % Roughly, we think of each~$V_n$ as an approximation for the sequences captured by the test. Thus, being captured by the test can be viewed as a new sense of obeying~$\cost$ that can be applied to ML-random sets.

The fundamental connection between our two uses of a cost function is the following:  
\begin{proposition}[\cite{Bienvenu.Greenberg.ea:16}, Prop.\ 4.2] \label{prop:basic fact} 
If $A \models \cost$ and $Y$ is an ML-random sequence captured by a $\cost$-bounded test, then $A \leT Y$.
\end{proposition}

To sketch the proof, say $Y$ is covered by the $\cost$-test $\seq{V_n}$. define a functional $\Gamma$ by letting $\Gamma^X(n) = A_s(n)$ if $X$ goes into $V_n$ at stage $s$. An $A$-change threatens to invalidate these definitions, so we build a Solovay test; if $A(n)$ is the least change at stage $s$, put $V_{n,s}$ into the test. Being ML-random, $Y$ is only in finitely many components of the Solovay test, so $\Gamma^Y(n) = A(n)$ for sufficiently large~$n$.

%% Joe: not clear what this is saying
% The idea of the proof (assuming that $A$ is c.e.) is to collect into a Solovay test $\seq{U_n}$ the oracles that become invalid through an $A$-change. If an enumeration of $A$ obeying $\cost$ changes $A(n)$ at stage $s$, then $U_{n,s}$ is listed as a component of the test. Being ML-random, $Y$ is outside $U_n$ for sufficiently large $n$, so $Y$ computes $A$ correctly on sufficiently large inputs~$n$.
 
As mentioned, Ku\v{c}era showed that every $\DII$ ML-random sequence is Turing above a non-computable c.e.\ set. Hirschfeldt and Miller in unpublished work dating from 2006 strengthened this: below any $\Sigma^0_3$ null class of randoms there is a non-computable c.e.\ set. Relying on \cref{prop:basic fact}, these proofs can be framed in the language of cost functions; see \cite{GreenbergNies:benign} and \cite[5.3.15]{Nies:book}, respectively.

% We write $\cost \to \dost$ ($\cost$ is stronger than $\dost$) if every set obeying $\cost$ also obeys $\dost$. 

\smallskip

The existence of an ML-complete $K$-trivial was shown in~\cite{Bienvenu.Greenberg.ea:16}. Given that $K$-trivials are the sets obeying $\cost_\Omega$, and the ML-randoms computing all $K$-trivials are the ones that fail some $\cost_\Omega$-test, ML--completeness for $K$-trivials coincides with being smart for $\cost_\Omega$ in the sense of the next definition:
% 
% where it is shown that (a) an ML-random is captured by a $\cost_\Omega$-test if and only if it computes all $K$-trivial sets; and (b) there is a $K$-trivial set~$A$ computable only from randoms that are captured by a $\cost_\Omega$-test. Thus, this set has greatest ML-degree among the $K$-trivials. A set of this type was called ``smart''. We extend this notion to designate a converse to \cref{prop:basic fact}. 

\begin{definition} \label{def:smartness}
 Let $\cost$ be a cost function and~$A$ be a $K$-trivial set. We say that~$A$ is \emph{smart for $\cost$} if~$A$ obeys~$\cost$ and for each ML-random set $Y$, 
\begin{center} $Y$ is captured by a $\cost$-bounded test $ \LR A \leT Y$. \end{center} 
\end{definition}

Informally, $A$ is as complex as possible for obeying $\cost$, in the sense that the only random sets $Y$ above $A$ are the ones that have to be there because of \cref{prop:basic fact}. To summarise the discussion in the introduction, in \cref{thm:smarts_exist}, we show that there is a smart set for any cost function~$\cost$ such that obedience to $\cost$ implies $K$-triviality. In Corollary~\ref{cor:smart ML complete}, we use this to show that if $\cost \to \cost_\Omega$ then $A$ is smart for $\cost$ iff $A$ is ML-complete for $\cost$. Dual to Theorem~\ref{thm:smarts_exist}, in \cref{prop:A_is_smart_for_c_A}, we prove that any $K$-trivial set is smart for some cost function $\cost_{(A)}$; when $A$ is c.e., this will be the strongest cost function obeyed by~$A$.

\section{Inherent enumerability of the \texorpdfstring{$K$}{K}-trivials up to \texorpdfstring{$\equiv_\ML$}{ML-equivalence}}
\label{sec:inherent_enumerability}
%%%%%%%%
%%%%%%%%

In this section, we considerably strengthen the result~\cite{Nies:AM} that every $K$-trivial is computable from a c.e.\ $K$-trivial: we show that the c.e.\ $K$-trivial can be taken to have the same ML-degree.
 This is a powerful tool. It is usually easier to prove results for the c.e.\ $K$-trivials; extra work is needed to lift results to the general case. \cref{thm:ML_degrees_ce_generated} simplifies this process in many cases. Indeed, we use it in both \cref{sec:strongest} and \cref{sec:proof} for this purpose. 
\begin{thm} \label{thm:ML_degrees_ce_generated}
For every $K$-trivial set $A$, there is a ($K$-trivial) c.e.\ set $D$ such that $D \geT A$ and $D \equiv_\ML A$. \end{thm}
%
%This goes a long way to formalising the intuition that $K$-triviality is essentially a c.e.\ notion. In particular, every $K$-trivial ML-degree contains a c.e.\ set. REpetitive

Note that the $K$-triviality of $D$ is free: every $K$-trivial is Turing below an incomplete ML-random sequence~\cite{CoveringProblemAnnouncement,DayMiller:Covering}, and every c.e.\ set below an incomplete random is $K$-trivial. So by virtue of being ML-equivalent to~$A$, the c.e.\ set~$D$ must be $K$-trivial. 
% Note that $D$ is necessarily $K$-trivial: combining \cite{Bienvenu.Greenberg.ea:16,DayMiller:Covering}, $A$ is below some incomplete random set (see \cite{CoveringProblemAnnouncement} for a survey of this), and, by the aforementioned result in~\cite{HirschfeldtNiesStephan:UsingRandomSetsAsOracles} any c.e.\ set $D$ below a Turing incomplete random set is $K$-trivial. The theorem

 \cref{thm:ML_degrees_ce_generated} follows from a fact of independent interest. Intuitively, the fact states that each $K$-trivial $A$ has a computable approximation that converges faster than any computation of~$A$ from a random.

\begin{lemma} \label{lem:randoms_compute_modulus}
For every $K$-trivial set $A$, there is a computable approximation $\seq{A_s}$ of~$A$ such that for every ML-random $X$ and Turing functional~$\Phi$ with $A = \Phi^X$ the following holds. For sufficiently large~$n$, if $A\uhr n \preceq \Phi^X_s$, then $A_t\uhr n = A\uhr n$ for every $t \ge s$.
\end{lemma}

\begin{proof}[Proof of \cref{thm:ML_degrees_ce_generated}]
Assuming that the lemma holds, we argue that a random computing~$A$ must also compute a modulus for~$A$; this modulus will have a c.e.\ degree. Let $\seq{A_s}$ be the approximation from the lemma. Let $D$ be the change-set for this approximation: $(n,k) \in D$ if and only if there is a sequence of stages $s_0 < \cdots < s_k$ with $A_{s_i}(n) \neq A_{s_{i+1}}(n)$ for all $i < k$. $D$ is clearly c.e., and it can compute $A(n)$ by searching for the least $k$ with $(n, k) \not \in D$ and considering the parity of $k$ and the value of $A_0(n)$.

Suppose that $A = \Phi^X$ for some random $X$. By the lemma, there is $N$ such that for all $n\ge N$, the approximation converges to $A\uhr n$ faster than $\Phi^X$ does. Thus $X$ can compute $D(n, k)$ by waiting until a stage $t$ with $A\uhr n \preceq \Phi^X_t$ and then only searching for sequences of stages $s_0 < \cdots < s_k$ such that $s_k \le t$. For $n< N$, we can arrange that our computation knows $D(n,k)$ by table lookup.
\end{proof}
For the rest of the paper, we fix a Turing functional~$\Upsilon$ that is universal in the sense that $\Upsilon^{{0^e1}\conc{X}} = \Phi_e^{X}$ for each~$X$ and~$e$. We assume that for every~$e$, for all sufficiently large~$n$, for all~$s$ and $X$, $\Phi_{e,s}(X;n)\converge \Rightarrow \Upsilon_s(0^e1\conc{X};n)\converge$. 
\begin{proof}[Proof of~\cref{lem:randoms_compute_modulus}] 
\NumberQED{lem:randoms_compute_modulus} It suffices to prove the lemma for the functional~$\Upsilon$. 
Let us first give a brief explanation of the proof. We will use the ``Main Lemma'' derived from the golden run construction~\cite[5.5.1]{Nies:book}. The Main Lemma says that if we design a left-c.e.\ oracle discrete measure on~$\w$, (equivalently, an adaptive additive cost function, or a prefix-free oracle machine), then there is a computable approximation $\seq{A_s}$ of~$A$ such that the total of all weights that are believed at some stage of the construction and later are shown to be false is finite. Roughly, we would like, at stage $s$, to put the weight $\leb (\Upsilon_s^{-1}[A_s\rest{n}])$ on the string $A_s\rest{n}$, where $\leb$ is Lebesgue measure on Cantor space, and 
\[
	\Upsilon_s^{-1}[\s] = \big\{ X\in 2^\w \,:\, \Upsilon_s^X \succeq \s \big\}.
\]
Such an approximation for~$A$ will be as required: we can put a Solovay test on the reals that compute~$A$ too early, and thus random oracles will only converge and agree with $A_s\rest{n}$ after it has settled. 

The problem is that this definition does not give a discrete measure: there is no reason to believe that $\sum_n \leb (\Upsilon^{-1}[A\rest{n}])$ is finite. What we notice, though, is that if an oracle~$X$ gives us a correct version of~$A$ too early, then this version $A_s\rest{n}=A\rest{n}$ will later change to $A_t\rest{n}\ne A\rest{n}$, but after that will need to change back. We can thus put the weight not on $A\rest{n}$ but on an incorrect version $A_t\rest{n}$. And this is guaranteed to give a measure: the collection of strings of the form $(A\rest{n})\conc (1-A(n))$ that disagree with~$A$ only on the last bit is pairwise incomparable, and so the preimages under~$\Upsilon$ of these strings are pairwise disjoint. 

\smallskip
We provide the formal details. For $\sigma \in 2^{<\w}$ with $\sigma \neq \langle\rangle$, define $\hat{\sigma}$ to be the binary string of the same length which disagrees with $\sigma$ on the final bit, but agrees on all other bits. For example, if $\sigma = 001011$, then $\hat{\sigma} = 001010$. For $\sigma \in 2^{<\w}$ with $\sigma \neq \langle\rangle$, for brevity, define
\[
\+U_\sigma = \left\{ X \ : \ \Upsilon^X \succeq \hat{\sigma}\right\}.
\]
To avoid the need of repeatedly dealing with $\langle\rangle$ separately, define $\+U_{\langle\rangle} = \emptyset$. Note that $(\+U_\sigma)_{\sigma \in 2^{<\w}}$ are uniformly $\Sigma^0_1$-classes. Also, for $\sigma \precneq \rho$, $\+U_\sigma$ and $\+U_\rho$ are disjoint.

For our argument, we will require a computable approximation to $A$ that obeys an adaptive cost function---a cost function where the cost at a given stage depends on the approximation up to that stage. For $\seq{A_t}$, a computable approximation of~$A$, define
\[
\cost^{\seq{A_t}}(n,s) = \leb(\+U_{A\uhr{n+1}}[s]).
\]

\begin{claim} \label{clm:modulus:adaptive}
There is a computable approximation $\seq{A_t}$ to $A$ that obeys $\cost$. That is, if $n_s$ is least with $A_s(n_s) \neq A_{s+1}(n_s)$, then $\sum_s \cost^{\seq{A_t}}(n_s,s ) < \infty$.
\end{claim}

\begin{proof}
\NumberQED{clm:modulus:adaptive}
Uniformly in $\sigma$ and $s$, let $C_{\sigma,s} \subset 2^{<\w}$ be a finite anti-chain that generates $\+ U_{\sigma,s}$, with $C_{\sigma, s} \subseteq C_{\sigma, s+1}$. Define an oracle machine $M$ with $M^{\sigma}_s(\pi)\converge$ for $\pi \in C_{\sigma,s}$. Since $\+ U_\sigma$ and $\+ U_\rho$ are disjoint for $\sigma \precneq \rho$, $M$ is prefix-free.

Note that for $s>n$ and any computable approximation $\seq{A_t}$,
\[
\sum_\pi 2^{-|\pi|} \bool{ M^A(\pi)[s]\converge \& \ \text{use}\, M^A(\pi)[s] = n+1} \ge \cost^{\seq{A_t}}(n,s).
\]
Fix a computable approximation $\dot{\seq{A_q}}$ to $A$. By the Main Lemma~\cite[5.5.1]{Nies:book} derived from the golden run construction, there is a computable sequence $q(0) < q(1) < \cdots$ with $q(0) \ge 1$, such that if we define $m_s$ to be least with $\dot{A}_{q(s)} \neq \dot{A}_{q(s+1)}$, then
\[
\sum_s \sum_\pi 2^{-|\pi|} \bool{ M^{\dot{A}}(\pi)[q(s)]\converge \& \ m_s < \text{use}\, M^{\dot{A}}(\pi)[q(s)] \le q(s-1)} < \infty.
\]
Now let $A_t = \dot{A}_{q(t)}$, so $n_s = m_s$. Since $s \le q(s-1)$, if $n_s < s$ then the inner summation above is at least $\cost^{\seq{A_t}}(n_s, s)$. So $\sum_s \cost^{\seq{A_t}}(n_s, s) < \infty$, as desired.
\end{proof}

%Since $\sum_{n} \leb(\+U_{Z\uhr n}) \le 1$ for every $Z \in 2^\w$, we can pair the Machine Existence Theorem with the main lemma of the golden run and conclude that there is a computable approximation $\seq{A_s}$ to $A$ that obeys $c$. That is, if $n_s$ is least with $A_s(n_s) \neq A_{s+1}(n_s)$, then $\sum_s c(n_s, s) < \infty$. This is our desired approximation.

Now, let $\seq{B_s}$ be a uniformly computable sequence of finite anti-chains with
\[
\bigcup_{\tau \in B_s} [\tau] = \+U_{A\uhr{n_s+1}}[s],
\]
and define $S = \bigcup_s B_s$. Note that
\begin{eqnarray*}
\sum_{\tau \in S} 2^{-|\tau|} &\le& \sum_s \sum_{\tau \in B_s} 2^{-|\tau|}\\
&=& \sum_s \leb(\+U_{A\uhr{n_s+1}}[s])\\
&=& \sum_s \cost^{\seq{A_t}}(n_s,s) < \infty.
\end{eqnarray*}
Thus $S$ is a Solovay test.

Let~$X$ be random and suppose that $A = \Phi_e^X$. Then $Y = 0^e1\conc{X}$ is random and $A = \Upsilon^{Y}$. Since $Y$ is not captured by $S$, we can fix an $s_0$ such that no $B_s$ with $s \ge s_0$ contains an initial segment of $Y$. Fix $N$ such that for all $n \ge N$, if $A\uhr n \preceq \Upsilon^Y_s$, then $s \ge s_0$. 

\begin{claim} \label{clm:modulus:success}
For all $n \ge N$, if $A\uhr n \preceq \Upsilon^Y_s$, then $A_t\uhr n = A\uhr n$ for every $t \ge s$.
\end{claim}

\begin{proof}
\NumberQED{clm:modulus:success}
Suppose $n \ge N$ were a counterexample. Let $s$ be such that $A\uhr n \preceq \Upsilon^Y_s$ and $t\geq s$ be such that $A_t\uhr n \neq A\uhr n$ and $A_{t+1}\uhr n = A\uhr n$. Note that definitionally, $n_t < n$. Since $A_t\uhr n_t = A_{t+1}\uhr n_t$, we know that $A_t\uhr n_t = A\uhr n_t \prec \Upsilon^Y_s$ and $A_t\uhr n_t+1 \neq A\uhr n_t+1$. So $\widehat{(A_t\uhr n_t+1)} = A_{t+1}\uhr n_t+1$, and $Y \in \+U_{A_t\uhr n_t+1}$. By assumption, $Y$ has already entered this $\Sigma^0_1$-class by stage $s$. Since $t\geq s$, $B_t$ contains an initial segment of $Y$, contrary to our choice of $N$ and $s_0$.
\end{proof}

Since for sufficiently large $n$, convergence of $\Phi^X_{e,s}$ up to $n$ implies convergence of $\Upsilon^Y_s$ up to $n$, the lemma follows.
\end{proof}

%%%%%%%%
%%%%%%%%
\section{For each cost function there is an ML-complete set}
\label{sec:smart}
%%%%%%%%
%%%%%%%%

In this section, we prove the existence of a smart set as in Definition~\ref{def:smartness} for each cost function $\cost$ that implies $K$-triviality. This yields in Corollary~\ref{cor:smart ML complete} the equivalence of smartness for $\cost$, and ML-completeness for the class of sets obeying $\cost$, and thereby the existence of an ML-complete for the class.

For cost functions~$\cost$ and $\dost$, one writes $\cost \to \dost$ if $A \models \cost $ implies $A \models \dost$ for every $\Delta^0_2$ set~$A$. By \cite[Thm.~3.4]{Nies:CalculusOfCostFunctions}, this is equivalent to $\limcost \ge^\times \underline{\dost}$, that is, $\limcost$ multiplicatively dominates $\underline {\dost}$ (we may assume $\limcost(x) >0$ for every~$x$).
%Note that it matters here to include all $\DII$ sets $A$ in the definition of $\to$, not only the c.e.\ ones. (E.g., let $\limcost(n) = 3^{-n}$ and $\underline {\dost}(n) = \tp{-n}$, then cost function implication $\cost \to \dost$ trivially holds for c.e.\ sets even though $\limcost \not \ge^\times \underline {\dost}$.)
Recall that a set obeys~$\cost_\Omega$ if and only if it is $K$-trivial. So all sets obeying a cost function~$\cost$ are $K$-trivial if and only if $\cost\to \cost_\Omega$. 

We start with two simple lemmas.

\begin{lemma} \label{lem:shift}
Suppose that $aY$ fails a $\cost$-bounded test $\bigcap_n V_n$, where $a \in \{0,1\}$. Then $Y$ fails a $\cost$-bounded test.
\end{lemma}
\begin{proof}
We may suppose $a=0$ and $X \in V_n$ implies $X(0) =0$. Then $\leb (T[V_n])= 2 \leb (V_n)$, where $T$ is the usual shift operator on Cantor space, and so $ \seq{ T [V_n]}$ is also a $\cost$-bounded test. Clearly $Y$ fails it.
\end{proof}

We recall~\cite{Nies:CalculusOfCostFunctions} that an \emph{additive} cost function is a cost function of the form~$\cost_\alpha(n,s) = \alpha_s-\alpha_n$, where~$\seq{\alpha_s}$ is an increasing approximation of a left-c.e.\ real~$\alpha$. So $\limcost_\alpha(n) = \alpha-\alpha_n$. Since~$\Omega$ is Solovay complete among the left-c.e.\ reals, every time we see an increase in~$\alpha$, we can cause a proportional and later increase in~$\Omega$. Thus:

\begin{lemma} \label{lem:Omega_is_slow}
If~$\cost_\alpha$ is an additive cost function, then $\cost_\Omega\to \cost_\alpha$.
\end{lemma}

% \begin{proof}
% 	Consider the ML-test $\seq {V_m}$ given by $V_m = \bigcup_{n\ge m} (\Omega_n-2^{-n}, \Omega_n+2^{-n})$.
% \end{proof}

In particular, $2^{-n}\le^\times \limcost_\Omega(n)$.

\begin{theorem} \label{thm:smarts_exist} \

\noindent 
Given a cost function $\cost $ such that $\cost \to \cost_\Omega$, one can uniformly obtain a c.e.\ set $A$ which is smart for~$\cost$.
\end{theorem} 
\begin{proof}
Recall that $\Upsilon$ is a ``universal'' Turing functional in the sense that $\Upsilon^{{0^e1}\conc{X}} = \Phi_e^{X}$ for all~$X$ and~$e$. We build $A$ and a $\cost$-test $\seq{\+ U_k} $ capturing any ML-random $Y$ such that $A = \Upsilon^Y$. This suffices for the theorem by \cref{lem:shift}. The tension in this construction is between trying to capture all reals computing~$A$, and keeping the measure of $\+U_n$ bounded by (a multiple of) $\limcost(n)$. The idea is for us to move~$A$ in case we see that too many oracles compute it. This needs to be done judiciously; we must ensure that~$A$ obeys~$\cost$. The basic idea, as in \cite{Bienvenu.Greenberg.ea:16}, is to charge the cost of changing~$A$ to the increase in the measure of the \emph{error set}, the set of oracles that have already been proven to be incorrect about~$A$. Since $\cost\to \cost_\Omega$, the increase in the error set is bounded by~$\cost$, and so we can catch our tail. %\andre{refers to cats? Cryptic}

\smallskip
We proceed to the details. It will be clear from the proof that the construction is uniform in $\cost$.

 We apply the usual language for strings: if $\sss <_L \tau $ we say that $\sss$ lies to the left of $\tau$, and $\tau$ to the right of $\sss$. By delaying computations from appearing in~$\Upsilon$ during the construction, we may assume that for all~$Y$ and~$s$, $\Upsilon^Y_s$ does not lie to the right of~$A_s$. We build a global ``error set'':
\[ \+ E_s = \left\{ Y \,:\, \Upsilon^Y_s \text{ lies to the left of }A_s \right\}.
% {\colon} \ex n\, [ \Upsilon^Y_s(n)\converge=0 \lland A_s(n) = 1 ]\}
\]
An enumeration of a number into~$A$ causes~$A$ to move to the right, and so potentially adds elements~$\+E$; no elements can ever leave $\+ E$. The basic idea, again, is that we enumerate a number~$x$ into~$A$ only when the cost $\cost(x,s)$ is smaller than the amount by which the measure of~$\+E$ will be increased.

We will ensure that at every stage~$s$, 
\[ \tag{$\diamond$} \leb (\+ U_{k,s}) \le \cost(k,s) + \leb ( \+ E_{s+1} - \+ E_k). \] 
By \cref{lem:Omega_is_slow}, $\leb (\+ E - \+ E_k) = \limcost_{\leb(\+E)}(k) \le^\times \limcost_\Omega(k)$, so as $\limcost_\Omega(k) \le^\times\limcost(k) $, the test $\seq{\+ U_k} $ is indeed a $\cost$-test. 
We reserve the interval $I_k = [\tp k, \tp{k+1} )$ for ensuring ($\diamond$). 

The construction of the $\cost$-test $\seq{\+ U_k}$ and the c.e.\ set~$A$ is as follows.
At stage $s > k$ we let
\[
	\+V_{k,s} = \left\{ Y \,:\, \Upsilon_s^Y\prec A_s \andd \Upsilon_s^Y \uhr 2^{k+1} \text{ is defined} \right\};
\]
and 
\[
	\+U_{k,s} = \bigcup_{t\in [k,s]} \+V_{k,t}. 
\] 
% \[ \+ U_{k,s} = \bigcup_{t<s} \{ Y \colon A_t \uhr {x_s+1} \preceq \Upsilon_t^Y \} - \+ E_{k}. \]
As $\+V_{k,s}\supseteq \+V_{k+1,s}$, we have $\+U_{k,s}\supseteq \+U_{k+1,s}$ so $\seq{\+U_k}$ is nested. Note that $\+V_{k,s}$ is disjoint from~$\+E_k$, for every $s$, hence $\+U_{k,s}$ is disjoint from~$\+E_k$.

Let $s>k$. We let $x_s = x_s(k) = \min (I_k - A_{s})$. If ($\diamond$) threatens to fail at~$s$, namely $\leb (\+ U_{k,s}) > \cost(k,s) + \leb ( \+ E_s - \+ E_k)$, we enumerate $x_s(k)$ into $A_{s+1}$. This causes $\+ U_{k,s}$ to go into $\+ E_{s+1}$. Since~$\+U_{k,s}$ is disjoint from~$\+E_k$, it follows in this case that $\leb(\+U_{k,s})\le \leb(\+E_{s+1} - \+E_k)$, and so ($\diamond$) holds at stage~$s$. (Now the cycle can repeat: during stages $t \ge s=1$, the class $\+ U_k$ is allowed to add measure up to $\cost(k,t)$ without any action necessary. If the measure added exceeds $\cost(k,t)$ another enumeration into $A$ will be needed.)

\smallskip

First we verify that $x_s$ always exists, that is, we enumerate at most $\tp k$ times for~$\+ U_k$. By \cref{lem:Omega_is_slow}, we may assume that $\cost(x,s) \ge \tp{-x}$ for $x < s$. (To be clear, here we are using the fact that if $\limcost =^\times \underline{\dost}$, then the same sets obey $\cost$ and $\dost$.) If we enumerate $x_s(k)$ into~$A_{s+1}$, then $\leb(\+ U_{k,s}) > 2^{-k} + \leb(\+ E_s - \+E_k)$. Since $\+U_{k,s} \cap \+E_k = \emptyset$, and $\+U_{k,s} \subseteq \+E_{s+1}$, it follows that $\leb(\+E_{s+1} - \+E_s) > 2^{-k}$. Since $\leb (\+E) \le 1$, this can happen at most $2^k$ times.

Recall that for all~$Y$ and~$s$, $\Upsilon^Y_s$ does not lie to the right of~$A_s$. Hence, if $A = \Upsilon^Z$ then $Z \in \bigcap_k \+ U_k$. It remains to verify that $A \models \cost$. If we enumerate $x_s(k)$ into~$A_{s+1}$, then 
\[
\leb (\+ U_{k,s} - \+ E_s) = \leb (\+ U_{k,s} - (\+ E_s - \+ E_k) ) \ge \leb (\+ U_{k,s} ) - \leb(\+ E_s - \+ E_k) > \cost(k,s ) \ge \cost (x,s).
\]
Since $ \+ U_{k,s} - \+ E_s \sub \+ E_{s+1} - \+ E_s$, we see that $\cost(x,s) < \leb(\+E_{s+1} - \+E_s)$. This implies that the total cost of the enumeration of $A$ is at most $\leb(\+E)\le 1$.
\end{proof}

ML-completeness for a cost function was defined in~\ref{def:MLComp}.
\begin{cor} \label{cor:smart ML complete} Suppose that $\cost$ is a cost function such that $\cost \to \cost_\Omega$. Let $A $ be a $\DII$~set. Then 
$A$ is smart for $\cost$ $\LR$ $A$ is ML-complete for $\cost$.
\end{cor} 
\begin{proof} 
($\RA$) Suppose that $A\leT Y$ for ML-random $Y$. Then some $\cost$-bounded test captures~$Y$. If $B \models \cost$, then $B \leT Y$ by the basic fact, \cref{prop:basic fact}.

\noindent ($\LA$) Let $\wt A$ be smart for $\cost$. If $A \leT Y$ for ML-random $Y$, then $\wt A \leT Y$, so $Y$ is captured by a $\cost$-bounded test.
\end{proof}

In particular, the ML-degree of a smart set $A$ for $\cost$ is uniquely determined by~$\cost$. In contrast, for each low c.e.\ set $A$, there is a c.e.\ set $B \not \leT A$ such that $B \models \cost$ \cite[5.3.22]{Nies:book}. If $A$ is smart for $\cost$, then $A \oplus B$ is also smart for $\cost$. As every $K$-trivial is low, the Turing degree of a set $A$ that is smart for $\cost$ is never uniquely determined by $\cost$.

%%%%%%%%
%%%%%%%%
\section{Each \texorpdfstring{$K$}{K}-trivial set is ML-complete for a cost function}
\label{sec:strongest}
%%%%%%%%
%%%%%%%%

Given a $K$-trivial set~$A$, we will define a cost function $\cost_{(A)}$ with $A \models \cost_{(A)}$ such that every random computing $A$ is captured by a $\cost_{(A)}$ test. In other words, we build $\cost_{(A)}$ in such a way that $A$ is smart for $\cost_{(A)}$. Furthermore, in case that $A$ is c.e., $\cost_{(A)}$ is the strongest cost function that $A$ obeys, in the sense that if $A\models\cost$, then $\cost_{(A)}\to\cost$. In the introduction we mentioned applications of this to the structure of ML-degree of $K$-trivials, such as showing that there is no minimal pair.

We will provide an example showing that $\cost_{(A)}$ may not behave in an overly nice way. We build a c.e.\ $K$-trivial~$A$ such that the class of sets obeying $\cost_{(A)}$ is not closed downward under $\leT$. In fact, in our example $T(A)\not\models \cost_{(A)}$, where $T(A)$ is the shift of $A$, obtained by deleting the first bit. 

As before, $\Upsilon$ denotes a universal Turing functional. Let~$A$ be $K$-trivial. The idea for defining $\cost_{(A)}$ is as follows. Suppose first that~$A$ is c.e., and let~$\seq{A_s}$ be an effective enumeration of~$A$ that obeys~$\cost_\Omega$ \cite{Nies:CalculusOfCostFunctions}. We want to define $\cost_{(A)}$ so that we can capture by a $\cost_{(A)}$-bounded test all the reals~$Z$ such that $\Upsilon^Z = A$. The natural test is $\+U_k = \bigcup_{s\ge k} \Upsilon_s^{-1}[A_s\rest{k}]$. So we define $\limcost_A(k) = \leb(\+U_k)$. Why does~$A$ obey this cost function? Since the approximation is left-c.e., $A$ does not have to pay for the measure of the oracles that compute $A\rest{k}$ correctly: these only appear after~$A\rest{k}$ has settled, and after that, all changes to $A$ are beyond~$k$. So $A$ only needs to pay for the measure of those reals~$Z$ that compute an incorrect version $A_s\rest{k}$. This price is bounded by the increase of the measure of the error set: those oracles that compute some string to the left of~$A$. Thus the total $A$-cost is the same as the total $A$-cost of an additive cost function, and hence bounded by the total $\cost_\Omega$-cost of this enumeration; but this was chosen to be finite. 

When $A$ is not c.e., we use a c.e.\ intermediary. Let us give the details of the definition. By~\cref{thm:ML_degrees_ce_generated}, fix a c.e.\ set $C\equiv_\ML A$ that computes~$A$; let~$\Psi$ be a Turing functional such that $A = \Psi^C$. Fix an enumeration $\seq{C_s}$ of~$C$ and an approximation $\seq{A_s}$ of~$A$ that witnesses $A \models \cost_\Omega$. By speeding up both $\Psi$ and our approximations, we may assume that $A_s\uhr{s} \preceq \Psi_s^{C_s}$ for every $s$. To unify our construction with the earlier discussion, we assume that if $A$ is c.e., then $C = A$ and $\Psi$ is the natural reduction with identity bounded use.
%By delaying $\Upsilon$, we may assume that $\Upsilon_s^Y(n) = 1 \Rightarrow C_s(n) = 1$ for every $s, n$ and $Y$.

Similarly to what we did above, we let 
%\[ \+ E_s = \left\{Y : \exists n \,\left[\Upsilon^Y_s(n)\converge=0 \andd C_s(n) =1\right]\right\} \]
\[
	\+E_s = \left\{ Y \,:\, \Upsilon^Y_s \text{ lies to the left of }C_s \right\}
\]
and
\[
	\+V_{x,s} = \left\{ Y \,:\, \Upsilon_s^Y\prec C_s \andd A_s\rest{x+1} \preceq \Psi_s^{\Upsilon_s^Y} \right\};
\]
we then let
\[
	\cost_{(A)}(x,s) = \leb \left( \bigcup_{x<t\le s} \+V_{x,t} \right).
\]
Note that $\cost_{(A)}$ is monotonic, as $\+V_{x,t}\supseteq \+V_{x+1,t}$. It satisfies the limit condition if~$A$ is non-computable: certainly for all~$x$, $\limcost_A(x)\le 1$. If $A\rest{k}$ has stabilised by stage~$s$, then $\limcost_A(s)\le \leb \left\{ Y \,:\, A\rest{k} \preceq \Psi^{\Upsilon^Y} \right\}$. Hence $\lim_x\limcost_A(x) \le \leb \left\{ Y \,:\, A = \Psi^{\Upsilon^Y} \right\}$; if~$A$ is non-computable, this is 0.
% \[ \cost_{(A)}(x,s) = \leb \left(\bigcup_{x< t\le s} \left\{ Y \colon A_t \uhr {x+1} \preceq \Psi_t(C_t \wedge \Upsilon_t^Y) \right\} - \+E_{x+1}\right).\]
% Here $C_t\wedge \Upsilon_t^Y$ is the longest common initial segment of these strings, that is, the longest initial segment of~$C_t$ which is computed at stage~$t$ by~$Y$ via~$\Upsilon$.

\begin{proposition} \label{prop:A_is_smart_for_c_A}
$A$ is smart for~$\cost_{(A)}$. 
\end{proposition}
\begin{proof}
First we show that~$A$ obeys $\cost_{(A)}$. In fact, the fixed approximation $\seq{A_s}$ witnesses this. Define an increasing approximation of the left-c.e.\ ``error real'' by
\[
\epsilon_s = \leb \left(\+ E_{s+1}\right).
\]
Suppose that $A_s(x) \neq A_{s+1}(x)$. For each $t\in (x,s]$ and every $Y\in \+V_{x,t}$, $\Upsilon_t^Y$ lies to the left of $C_{s+1}$, and so $Y\in \+E_{s+1}$; on the other hand $Y\notin \+E_t$ and so $Y\notin \+E_{x+1}$. It follows that \[
	\cost_{(A)}(x,s) \le \leb \left(\+E_{s+1}\setminus \+E_{x+1}\right) = \cost_{\epsilon}(x,s).
\]
By \cref{lem:Omega_is_slow}, $\cost_\Omega\to \cost_\epsilon$, and so 
\[
	\cost_{(A)} \seq{A_s} \le \cost_\epsilon \seq{A_s} \le^\times \cost_\Omega\seq{A_s},
\]
and we assumed that the latter is finite. 

\smallskip

Next we show that every random real that computes~$A$ is captured by some $\cost_{(A)}$-bounded test. Since $C\le_\ML A$, every such real computes~$C$. By \cref{lem:shift}, it suffices to build a $\cost_{(A)}$-test capturing any random~$Y$ such that $C = \Upsilon^Y$. The desired test is the test $\+U_k = \bigcup_{s>k} \+V_{k,s}$ defined above. (Again, we assume that we delay computations, so for all~$Y$ and~$s$, $\Upsilon^Y_s$ does not lie to the right of~$C_s$.)
\end{proof}

As promised, in the case that $A$ is c.e., $\cost_{(A)}$ is the strongest cost function that $A$ obeys. In particular, $\cost_{(A)} \to \cost_\Omega$.

\begin{proposition} \label{prop:c_A_is_strongest}
Suppose that~$A$ is c.e.\ For any cost function $\cost$ such that~$A\models\cost$, we have $\cost_{(A)} \to \cost$.
\end{proposition}

\begin{proof}
After multiplying $\cost$ by a constant, we may assume that $\limcost(0) < 1/2$. Fix a computable speed-up $f$ such that $\cost\seq{A_{f(s)}} < 1/2$ (again, see \cite{Nies:CalculusOfCostFunctions}). Define a Turing functional~$\Gamma$ such that at every stage $t$,
\[
\leb\left(\{Y : A_{f(t)}\uhr{x+1} \prec \Gamma^Y_{t} \} - \+E_{\Gamma,t}\right) = \cost(x,t),
\] where $\+E_{\Gamma,t} = \{ Y : \Gamma^Y_t \text{ lies to the left of }A_{f(t)}\}$. By a simple argument $\leb(\+E_{\Gamma,t}) \le \cost\seq{A_{f(s)}} < 1/2$ for every $t$, so this construction may proceed.

Fix $e$ with $\Phi_e = \Gamma$. Then
\begin{align*}
\limcost_A(x) &= \leb\left( \bigcup_{x < t} \+V_{x,t}\right)\\
&= \leb\left( \bigcup_{x < t} \left\{ Y : A_t\uhr{x+1} \preceq \Upsilon_t^Y\right\}\right)\\
&\ge \leb\left\{ Y : A\uhr{x+1} \preceq \Upsilon^Y\right\}\\
&\ge 2^{-(e+1)}\cdot \leb\left\{ Y : A\uhr{x+1} \preceq \Upsilon^{0^e1\conc Y}\right\}\\
&\ge 2^{-(e+1)} \limcost(x).\qedhere
\end{align*}
\end{proof}
We provide the promised applications to the ML-degrees.
\begin{cor} \label{rem: ML structure} 
\noindent (a) There is no minimal pair in the ML-degrees of $K$-trivials.
 
\noindent (b) The  ML-degree of a noncomputable $K$-trivial  never contains a maximal Turing degree.
\end{cor}
\begin{proof} (a) Given noncomputable $K$-trivials $A,B$, let $D$ be a noncomputable set obeying the cost function $\cost_{(A)} + \cost_{(B)}$. Then $D \le_\ML A,B$ by Proposition~\ref{prop:basic fact}. 

(b) Suppose $A$ is in the ML-degree. By \cref{thm:ML_degrees_ce_generated} we may assume that $A$ is  c.e.  Some  c.e.\ $K$-trivial $B\not \leT A$   obeys $\cost_{(A)}$ by~\cite[5.3.22]{Nies:book}; then $A \oplus B \equiv_\ML A$. \end{proof}

Recall that $T(A)$ is the shift of $A$, which is obtained by deleting the first bit.

\begin{thm} \label{thm:shift}
For every cost function $\dost$ there is a cost function $\cost \ge \dost$ and a c.e.\ set $A$ such that $A \models \cost$ and $T(A) \not \models \cost$. 
\end{thm}

Since $\cost_{(A)} \to \cost$, this shows that $T(A) \not \models \cost_{(A)}$. Thus $\cost_{(T(A))} \not \to \cost_{(A)}$. 
In contrast, for each ML-random $Y$,
\smallskip
\begin{center} $Y$ fails some $\cost_{(A)}$ test $\LR$ $Y \geT A$ $\LR$ $Y \geT T(A)$ $\LR$ $Y$ fails some $\cost_{(T(A))} $-test \end{center} 
\smallskip
(see \cref{def:cost-bounded_test} for $\cost$-tests). So we have a pair of inequivalent cost functions that determine the same randomness notion.
 
\begin{proof}
The main idea is to enumerate the set~$A$ and the cost function~$\cost$ so that it has ``sudden drops'': numbers~$x$ with $\limcost(x)$ much smaller than $\limcost({x-1})$.

Let $\seq{B^0_t}, \seq{B^1_t}, \dots$ be a listing of all (possibly partial) computable enumerations. In particular, let $\seq{D_n}$ be an effective listing of the finite sets, and let $B^e_{t+1} = B^e_t \cup D_{\phi_e(t+1)}$, where defined.

At a stage $s$, we may declare $\cost(s-1,s) \ge \alpha$ for some dyadic rational $\alpha$, which by monotonicity entails that $\cost(y,t) \ge \alpha$ for each $y < s$ and $t \ge s$. At the end of stage $s$, we will define $\cost(x,s)$ for every $x < s$ to be the least value consistent with all of our declarations and also with $\cost(x,s) \ge \dost(x,s)$.

We must meet the global requirements that $\cost$ has the limit condition and that $A \models \cost$. We must also meet the requirements
\[
R_e \colon \, T(A) = \bigcup_t B^e_t \Rightarrow \cost\seq{B^e_t} \ge 1.
\]

The strategy for $R_e$ seeks to find an $x$ and an $s$ where $\cost(x-1,s)$ is large and $x-1\not \in B^e_s$. Then it enumerates $x$ into $A$ and waits until it sees a $t > s$ with $x-1 \in B^e_t$. This will increase $\cost\seq{B^e_t}$ by at least $\cost(x-1,s)$. Then the strategy seeks to repeat the process with a new $x$, continuing until $\cost\seq{B^e_t} \ge 1$.

To ensure that $\cost$ has the limit condition, we will give $R_e$ a bound $\alpha_e$ beyond which it is not allowed to increase $\cost$. This bound will also ensure that $R_e$ does not interfere with $R_{e'}$ for $e' < e$. To ensure that $A \models \cost$, we will not allow $R_e$ to cause enumerations with total cost exceeding $2^{-e}$. Other than a discussion of $\alpha_e$, our full strategy for $R_e$ is:
\begin{enumerate}
\item Let $s$ be the current stage. Declare $\cost(s-1,s) \ge \alpha_e$.
\item At stage $s+1$, declare $\cost(s,s+1) \ge 2^{-e}\cdot\alpha_e$.
\item Wait for a stage $u > s$ when one of the following happens:
\begin{enumerate}
\item If $\cost(s,u) > 2^{-e}\cdot\alpha_e$, return to Step (1).
\item If $s$ is enumerated into $A$, return to Step (1).
\item If $B^e_s$ converges with $s-1 \not \in B^e_s$, enumerate $s$ into $A$ and proceed to Step (4).
\end{enumerate}
\item Wait until $B^e_r$ converges for some $r > s$ with $s-1 \in B^e_r$.
\item If $\cost\seq{B^e_t}_{t = 0}^r \ge 1$, terminate the strategy. Otherwise, return to Step (1).
\end{enumerate}
Note that case (3a) might occur because of the actions of some other strategy, or might instead occur because of $\cost(s,u) \ge \dost(s,u)$. The latter can occur only finitely many times, because $\dost$ satisfies the limit condition.

Note also that if we reach Step (5), then $s-1 \not \in B^e_s$, $s-1 \in B^e_r$, and $\cost(s-1,s) \ge \alpha_e$, so $\cost\seq{B_t^e}_{t = 0}^r - \cost\seq{B_t^e}_{t = 0}^s \ge \alpha_e$. Thus we will reach Step (5) at most $1/\alpha_e$ times before meeting the requirement and terminating the strategy. Each enumeration has a cost of $2^{-e}\cdot \alpha_e$ by construction, and so the total cost of enumerations by this strategy is at most $2^{-e}$.

If the strategy waits forever at Step (3), then either $\seq{B^e_t}$ is partial, or $s \not \in A$ but $s-1 \in \bigcup_t B_t^e$, meaning we satisfy $R_e$ by negating the hypothesis. It thus remains only to show that we do not return to Step (1) via case (3a) or (3b) infinitely many times.

We wish to ensure that no $R_{e'}$-strategy for $e' > e$ can increase $\cost(s,u)$ beyond $2^{-e}\cdot \alpha_e$. So we define $\alpha_0 = 1$, $\alpha_{e+1} = 2^{-e}\cdot \alpha_e$. Now case (3a) cannot be caused by the action of any $R_{e'}$-strategy for $e' > e$. Nor can case (3b), because of our action at Step (2). It is then a simple induction that no strategy returns to Step (1) more than finitely many times.
\end{proof}
 
 %%%%%%%%%%%%%%%%%%%%%%%%%%%%%%
 %%%%%%%%%%%%%%%%%%%%%%%%%%%%%%
\iffalse \section{ML-reducibility versus domination notions of cost functions}
 
 If $\limcost \ge^\times \underline {\mathbf d}$ (equivalently, $\cost \to \dost$), then each $\dost$-bounded test is also $\cost$-bounded, so that a set passing all $\cost$-bounded tests passes all $\dost$-bounded tests. The converse does not hold: apply \cref{thm:shift} for $\dost= \cost_\Omega$ and obtain a c.e.\ $K$-trivial set $A$ such that $\cost_{T(A)} \not \to \cost_{A}$. On the other hand $A$ is smart for $\cost_{(A)} $ and $T(A)$ is smart for $\cost_{T(A)}$, so each set passing all $\cost_{T(A)}$-bounded tests passes all $\cost_{A}$-bounded tests.
 
 We want an algebraic condition equivalent to the implication of corresponding randomness notions. We say that $\cost$ weakly implies $\mathbf d$, written 
 $\cost \to_w \mathbf d$, if there is a computable function $f$ obeying $\cost$ such that $ \limcost(x) \ge^\times \underline {\mathbf d}(f(x)) $ for each $x$. For instance, in the example above $\cost_{(A)} \to_w \cost_{T(A)}$ via $f(x) = x+1$.

 \begin{conjecture} The following are equivalent for c.f.\ $\cost, \mathbf d$ implying $\cost_\Omega$.
 
 (i) Each set passing all $\cost$-bounded tests passes all $\mathbf d$-bounded tests
 
 (ii) $\cost \to_w \mathbf d$. \end{conjecture}
 
 (ii) $\to$ (i) is not hard: 
 
\fi
%%%%%%%
%%%%%%%
% A criterion for comparability of fragments of $\Omega$ under ML$^*$-reducibility
\section{\texorpdfstring{$K$}{K}-trivial sets Turing below fragments of~\texorpdfstring{$\Omega$}{Omega}}
\label{sec:comparability_of_fragments}
%\andre{changed title because need to keep focus on main theme of paper which is now the complexity of $K$-trivials under ML}
%\section{The computational strength of fragments of~\texorpdfstring{$\Omega$}{Omega}} 
%%%%%%%%
%%%%%%%%

%In this section, we introduce ML$^*$-reducibility (\cref{def:ML_star_reducibility}), a natural dual of ML-reducibility that compares Martin-L\"of random sequences according to the K-trivial sets that they compute. We use this notion to compare the relative strength of fragments of $\Omega$, by which we mean the restriction of the bit-sequence $\Omega$ to a computable set $R$ of locations. \cref{thm:criterion} gives a characterisation of the ML$^*$-strength of such a fragment based on the growth of the function $m \mapsto|R\cap m|$. \cref{thm:Omega_R_is_feeble} characterises the $K$-trivial sets computable from a fragment of $\Omega$ using an appropriate cost function.
%
%
%

\noindent\emph{Previous research.} %plssss don't delete this :( Andre This praises a paper that has already been published. Anyway it is there
We begin by discussing in some detail a theorem that motivated the present results. For a set $Z \sub \NN$ thought of as a bit sequence and an infinite set $R\sub \NN$, we denote by $Z_R$ the sequence obtained by erasing the bits of~$Z$ in locations outside of~$R$. If $1\le j \le n$, then the $j\tth$ \emph{$n$-column} of~$Z$ is $Z_{(j-1+n\Nat)}$. 
A set~$A$ is a \emph{$k/n$-base} if it is computable from the join of any~$k$ of the $n$-columns of some random sequence~$X$, in all possible ways. 

For a computable real $p$ such that $0< p \le 1$, let $\cost_{\Omega,p}(x,s) = (\Omega_s-\Omega_x)^p$. As mentioned in the introduction, three of the authors of the present paper proved:

\begin{theorem}[\cite{SubclassesPaper}] \label{thm:k_n-bases}
The following are equivalent for a set~$A$ and $1\le k< n$:
\begin{enumerate}
\item $A$ is a $k/n$-base.
\item $A$ is a $k/n$-base witnessed by $\Omega$, i.e., it is computable from the join of \emph{any}~$k$ of the $n$-columns of~$\Omega$.
\item $A$ obeys $\cost_{\Omega,k/n}$. % Joe removed ``is $K$-trivial''
\end{enumerate}
\end{theorem}

Hence the $p$-bases are characterised by cost functions. \Cref{thm:smarts_exist} implies that, for every rational $p\in (0,1)$, there is a smart~$p$-base: a greatest ML-degree of $p$-bases. If $p<q$, then every $p$-base is also a $q$-base, as $\cost_{\Omega,p}\ge^\times \cost_{\Omega,q}$. However there also is a $q$-base that is not a $p$-base. Thus, the smart $p$-bases form a dense chain of ML-degrees.

Using Theorem~\ref{thm:criterion} below, we can add fourth equivalence to \Cref{thm:k_n-bases}, one that appears to be significantly weaker than (2):

\begin{enumerate}\it
\setcounter{enumi}{3}
\item $A$ is $K$-trivial and is computable from the join of \emph{some} choice of~$k$ of the $n$-columns of~$\Omega$.
\end{enumerate}
In other words: if a $K$-trivial is computable from \emph{some} $k/n$-fragment of $\Omega$, then it is computable from \emph{any} $k/n$-fragment of $\Omega$. Recall that any c.e.\ set computable from a Turing incomplete random set is $K$-trivial~\cite{HirschfeldtNiesStephan:UsingRandomSetsAsOracles}. Since every $k/n$-fragment of $\Omega$ is incomplete, we obtain:

\begin{cor}
If~$X$ and~$Y$ are both $k/n$-fragments of $\Omega$ for $k<n$, then $X$ and~$Y$ compute the same c.e.\ sets.
\end{cor}

% \begin{theorem} \label{thm:c.e._p_base_single}
% 	Let $1\le k\le n$. The following are equivalent for a c.e.\ set~$A$:
% 	\begin{enumerate}
% 		\item $A$ is computable from $\Omega_{R(T,n)}$ for some $T\subseteq \{1,2,\dots, n\}$ of size~$k$. 
% 		\item $A$ is computable from $\Omega_{R(T,n)}$ for all $T\subseteq \{1,2,\dots, n\}$ of size~$k$. 
% 	\end{enumerate}
% \end{theorem}
% That is, for a c.e.\ set to be a $k/n$-base, it is sufficient to be computable from the join of just one choice of $k$-many $n$-columns.

\noindent In particular, if a c.e.\ set is computable from one half of~$\Omega$, it is also computable from the other half.

\medskip
\noindent \emph{The cost functions $\cost_R$.}
We now turn to the general analysis of the question which $K$-trivials are computed by fragments of $\Omega$. For $n\ge 1$ and $T\subseteq \{1,2,\dots,n\}$, let $R(T,n) = \bigcup_{j\in T} j-1+n\Nat$. So $\Omega_{R(T,n)}$ is the join of the $n$-columns of~$\Omega$ indexed by~$T$ (up to a simple computable permutation, depending on how we take the join).

Let $R$ be an infinite computable set. The first question is how to generalise the cost function $\cost_{\Omega,k/n}$ to a cost function $\cost_{\Omega,R}$. 
A basic step in the analysis of $k/n$-bases was the observation that if $T\subseteq \{1,2,\dots, n\}$ has size~$k$, then $\Omega_{R(T,n)}$ is captured by a $\cost_{\Omega,k/n}$-test; this gave the implication (3)$\then$(2) of \cref{thm:k_n-bases}. We would like to capture the bits of $\Omega_R$ that are given by $\Omega\rest{n}$ by the $n\tth$ component of a $\limcost_{\Omega,R}$-test. So perhaps the first guess would be to define $\limcost_{\Omega,R}(n) = (\Omega-\Omega_n)^{|R\cap n|/n}$. It turns out that this is not quite right; it works if $R = R(T,n)$, but that is misleading because in that case the density of initial segments of $R$ is more or less constant $k/n$. What would work is $\limcost_{\Omega,R}(n) = (\Omega-\Omega_n)^{|R\cap k(n)|/k(n)}$, where $\Omega-\Omega_n \in (2^{-k(n)-1}, 2^{-k(n)}]$. However, it is not clear that this cost function will be monotonic if the density of~$R$ varies. We get around this technical complication by using a ``discrete'' version.

For $n<\w$, let
\[
	k(n) = \floor{-\log_2 (\Omega-\Omega_n)},
\]
so $2^{-k(n)-1}< \Omega-\Omega_n\le 2^{-k(n)}$. Define $k_s(n)$ similarly, replacing $\Omega$ by $\Omega_s$. Note that $k(n)\le n$ for all but finitely many~$n$ (otherwise, $\Omega$ would not be random).

\begin{definition} For an infinite computable $R \subseteq \w$, define
\[
	\cost_{\Omega, R}(n,s) = 2^{-|R\cap k_s(n)|}. 
\] \end{definition}
\noindent The cost function $\cost_{\Omega, R}$ is monotonic: $k_{s+1}(n)\le k_s(n)$ and $k_s(n+1)\ge k_s(n)$. It also satisfies the limit condition: $\limcost_{\Omega,R}(n) = 2^{-|R\cap k(n)|}$ is finite and, since $\lim_n k(n) = \infty$ and~$R$ is infinite, $\lim_{n}\limcost_{\Omega,R}(n) = 0$. Finally, we note that $\cost_{\Omega, R}(n,s)\cdot \cost_{\Omega, R^\complement}(n,s) = 2^{-k_s(n)} =^\times \Omega_s - \Omega_n$, where $R^\complement$ is the complement of~$R$. 

\begin{remark} \label{rmk:old_p_cost}
For any infinite $R$, note that $\limcost_{\Omega,R}(n) =^\times (\Omega-\Omega_n)^{|R\cap k(n)|/k(n)}$. For,
\begin{align*}
\left(2^{-k(n)-1}\right)^{|R\cap k(n)|/k(n)} &< (\Omega-\Omega_n)^{|R\cap k(n)|/k(n)} \le \left(2^{-k(n)}\right)^{|R\cap k(n)|/k(n)}\\
2^{-|R\cap k(n)|}\cdot 2^{-|R\cap k(n)|/k(n)} &< (\Omega-\Omega_n)^{|R\cap k(n)|/k(n)} \le 2^{-|R\cap k(n)|}\\
\frac12\cdot \limcost_{\Omega,R}(n) &< (\Omega-\Omega_n)^{|R\cap k(n)|/k(n)} \le \limcost_{\Omega,R}(n).
\end{align*}
In particular, if $T\subseteq \{1,2,\dots, n\}$ has size~$k$, then $\limcost_{\Omega,R(T,n)}=^\times \limcost_{\Omega,k/n}$. 
\end{remark}
\begin{proposition} \label{prop:antichain:capturing}
$\Omega_R$ is captured by a $\cost_{\Omega, R}$-test. 
\end{proposition}
\begin{proof}
We use the idea from the proof of~\cite[Prop.\ 2.9]{SubclassesPaper} that each $p$-Oberwolfach test is covered by a $\cost_{\Omega,p}$-tests. 
%For each $\s\in 2^{<\w}$, let 
%\[
%	G_\s = \left\{ X\in 2^\w \,:\, (\forall m< |R\cap |\s|| ) \,\, X(m) = \s( p_R(m)) \right\},
%\]
%where $p_R(m)$ is the $m\tth$ element of~$R$. 

For a string $\s\in 2^{<\w}$ of length $t$, let $\s_R$ denote the string of length $|R \cap t|$ given by the bits of $\s$ with location in $R$. Let $G_\s = \left\{ X\in 2^\w \,:\, \s_R \prec X\right \}$.
Clearly $\leb(G_\s)= 2^{-|R\cap t|}$. 

Let $U_n = \bigcup_{s\ge n} G_{\Omega_s\rest{n}}$. Then $\Omega_R\in \bigcap_n U_n$. Note also that $U_{n+1}\subseteq U_n$. Since $\Omega-\Omega_n \le 2^{-k(n)}$, the set $\{\Omega_s\rest{k(n)}\,:\, s\ge n\}$ contains at most two strings. But $U_n \subseteq \bigcup_{s\ge n} G_{\Omega_s\rest{k(n)}}$, so
\[
	\leb(U_n) \le 2\cdot 2^{-|R\cap k(n)|} = 2\cdot \limcost_{\Omega,R}(n). \qedhere
\]
\end{proof}

We proceed to the main theorem of this section. The equivalence (i)$\leftrightarrow$(iii) provides a simple combinatorial description of when a fragment $\Omega_S$ computes no more $K$-trivials than another fragment $\Omega_R $: for each number~$m$, the size of $S$ below $m$ exceeds the size of $R$ below $m$ by at most a constant. 

\begin{thm} \label{thm:criterion}
The following are equivalent for infinite computable sets~$R$ and~$S$:
\begin{itemize}
\item[(i)] Each $K$-trivial computed by $\Omega_S$ is also computed by $\Omega_R$.
%$\Omega_S \le_{\ML^*} \Omega_R$. 
\item[(ii)] $\Omega_R$ is captured by a $\cost_{\Omega,S}$-test.
\item[(iii)] $|S\cap m| \le^+ |R\cap m|$.
\item[(iv)] $\cost_{\Omega,S} \to \cost_{\Omega,R}$.
\end{itemize}
\end{thm}

\noindent The promised application (from the beginning of the section) of the main result follows easily: if $T,T'\subseteq \{1,2,\dots, n\}$ have size~$k$, then $|R(T,n)\cap m|=^+ |R(T',n)\cap m|$. Therefore, any $k$-trivial computable from $\Omega_{R(T,n)}$ is computable from every $k/n$-fragment of~$\Omega$, hence is a $k/n$-base. 
%\andre{Andre: the interesting case, especially compared to the previous paper, is when the R, S are not evenly distributed. This gets totally lost in this exposition. The reader will think we only care about the R(T,n) sets, and wonder why we state it in more general form. FIXED by adding to the initial paragraph of the section}

\begin{remark}[The dual of ML-reducibility] \label{rem:ML*} The relative complexity of fragments of~$\Omega$ can be understood in terms of the dual of ML-reducibility.
 For general ML-random sequences~$Y$ and~$Z$, we write $Y\le_{\ML^*} Z$ if for every $K$-trivial set~$A$, if $A\leT Y$ then $A\leT Z$.

Again Turing reducibility implies ML$^*$-reducibility. The top degree consists of those randoms that compute all $K$-trivial sets; these are the randoms that fail some $\cost_\Omega$-test (i.e., the non-\emph{Oberwolfach randoms}~\cite{Bienvenu.Greenberg.ea:16}). Of course, these include all the Turing complete randoms. The bottom degree consists of the weakly 2-random sequences, the randoms that compute no $K$-trivial sets. 

Using this language, (i) in Theorem~\ref{thm:criterion} states that $\Omega_S \le_{\ML^*} \Omega_R$. The equivalence (i)$\lra$(iii) in~\ref{thm:criterion} provides a complete characterisation of ML${}^*$-reducibility between fragments of~$\Omega$ by a simple combinatorial condition on the underlying computable sets. 
%In particular, for $1\le k\le n$, any two $k/n$-fragments of~$\Omega$ have the same ML$^*$-degree. 
The intuition is that as $R$ gets thinner, $\Omega_R$ gets computationally weaker (in the coarse sense of ML${}^*$). The randomness enhancement principle says that among ML-random sets, being computationally weaker is equivalent to being more random (see \cite{Nies:ICM} and the discussion there). By this principle, $\Omega_R$ also gets more random as $R$ gets thinner. 
 \end{remark}

Due to the relative length of the proof, we state and prove the implication (ii)$\to$(iii) of \cref{thm:criterion} separately.

\begin{proposition} \label{prop:antichain:non-capturing} \

%Let $R$ be a computable set such that $\limsup_m 2|R\cap m| - m = \infty$. Then $\Omega_{R^\complement}$ is $\cost_{\Omega, R}$-random. 
\noindent Let $R$ and $S$ be infinite computable sets such that $|S\cap m| \nle^+ |R\cap m|$, i.e., the function $m \mapsto |S\cap m|- |R\cap m|$ is unbounded. Then $\Omega_R$ is $\cost_{\Omega,S}$-random.
\end{proposition}
\begin{proof}
Suppose for a contradiction that $\Omega_{R}$ can be captured by a $\cost_{\Omega, S}$-test. Then using Proposition~\ref{prop:antichain:capturing} there is a nested test $\seq{\+U_n}$ capturing the sequence $Z=\Omega_R \oplus \Omega_{R^\complement}$ with $\leb(\+U_{n})\le (\limcost_{\Omega,S}(n))\cdot (\limcost_{\Omega,R^\complement}(n))$. We show that this implies that~$Z$ is not ML-random. To do this, we show how to uniformly enumerate an open set~$\+V$ of small measure that contains~$Z$.
	
Note that $|S\cap m| - |R\cap m| = |S\cap m| + |R^\complement\cap m| - m$. Given a rational $\epsilon > 0$, we can thus effectively find a $k$ with $|S \cap k| + |R^\complement\cap k| - k > 1- \log \epsilon$.

Define a location $n_s>k$ recursively at stages $s\ge k$. Recall that $k_s(n) = \floor{-\log_2(\Omega_s-\Omega_n)}$. 	For $s = k$, or if $k_{s+1}(n_s)< k$, we let $n_{s+1}$ be the least $n\ge k$ such that $k_{s+1}(n) >k$. Otherwise, we let $n_{s+1}= n_s$. 

There are at most $2^{k+1}$ stages~$s$ at which $n_s\ne n_{s-1}$. For let $s<t$ be two such stages, then $\Omega_s - \Omega_{n_s} \le 2^{-(k+1)}$. But $\Omega_t - \Omega_{n_s} > 2^{-k}$, so $\Omega_t - \Omega_s > 2^{-(k+1)}$.
%\andre{Dan pls check algebra. maybe issue with k versus k-1}
%Dan: It's correct.

Let $\+V = \bigcup_{s>k} \+U_{n_s,s}$. For each stage~$s\ge k$, we have $k_s(n_s)\ge k$, so
\begin{multline*}
\leb(\+U_{n_s,s}) \le (\cost_{\Omega,S}(n_s,s))\cdot (\cost_{\Omega,R^\complement}(n_s,s)) = \\
2^{-|S\cap k_s(n_s)|}\cdot 2^{-|R^\complement\cap k_s(n_s)|}
\le 2^{-|S\cap k|}\cdot 2^{-|R^\complement\cap k|} = \\
 2^{-(|S\cap k| + |R^\complement\cap k|)}
< 2^{\log \epsilon - k-1} = 2^{-k-1}\cdot \epsilon.
\end{multline*}
Therefore,
\[
	 \leb(\+V) \le 2^{k+1}\cdot 2^{-k-1}\cdot \epsilon = \epsilon.\qedhere
\]
% Suppose for a contradiction that $\Omega_{R^\complement}$ can be captured by a $\cost_{\Omega, R}$-test. Then there is a nested test $\seq{U_n}$ capturing the sequence $Z=\Omega_R \oplus \Omega_{R^\complement}$ with $\leb(U_{n})\le (\limcost_{\Omega,R}(n))^2$. We show that this implies that~$Z$ is not ML-random. To do this we show how uniformly, we can enumerate an open set~$V$ of small measure containing~$Z$.

% Given $\epsilon > 0$, we can effectively find a $k$ with $|R \cap k| - \frac12k > 1-\log \epsilon$.
% Define a location $n_s>k$ recursively at stages $s> k$. For $s = k$, or if $k_{s+1}(n_s)< k$, then we let $n_{s+1}$ be the least $n\ge k$ such that $k_{s+1}(n) >k$. Otherwise, we let $n_{s+1}= n_s$. 

% There are at most $2^{k+1}$ many stages~$s$ at which $n_s\ne n_{s-1}$. For let $s<t$ be two such stages. Then $\Omega_s - \Omega_{n_s} \le 2^{-(k+1)}$, but $\Omega_t - \Omega_{n_s} > 2^{-k}$, so $\Omega_t - \Omega_s > 2^{-(k+1)}$. 

% Let $V = \bigcup_{s>k} U_{n_s,s}$. For each stage~$s$, $k_s(n_s)\ge k$, so
% \[ \leb(U_{n_s,s}) \le (\cost_{\Omega, R}(n_s,s))^2 = 2^{-2|R\cap k_s(n_s)|} \le 2^{-2|R\cap k|} < 2^{-2(\frac12k + 1 - \log \epsilon)} = 2^{-k-2}\cdot \epsilon^2;
% \]
% we conclude that 
% \[
% 	 \leb(V) \le 2^{k+1}\cdot 2^{-k-2}\cdot \epsilon^2 = \frac12\epsilon^2 < \epsilon.\qedhere
% \]
\end{proof}

\medskip
The hardest implication is (iv)$\to$(i). To prove it, we rely on a lemma of interest on its own. Informally, the lemma says that if $X\oplus Y$ is ML-random, but $X$ is not too random in the sense that $X$ fails a $\cost_{\Omega, R^\complement}$-test for a co-infinite computable set $R$, then any $K$-trivial Turing below the ``other side'' $Y$ obeys the complementary cost function $\cost_{\Omega, R}$. For example, let $X= \Omega_{R^\complement}$ and $Y= \Omega_R$, so that in addition $Y$ fails a $\cost_{\Omega, R}$-test. In this case, any $K$-trivial set $A$ obeying $\cost_{\Omega, R}$ is below $Y$; the lemma says that these are the only $K$-trivials below $Y$. (See Theorem~\ref{thm:Omega_R_is_feeble}.)
\newcommand{\SAVEText}{Let $R \subseteq \w$ be computable and co-infinite. Suppose that $X \oplus Y$ is ML-random, and that~$X$ is captured by a $\cost_{\Omega, R^\complement}$-test. Suppose that~$A$ is $K$-trivial, and that $A \leT Y$. Then $A$ obeys $ \cost_{\Omega, R}$.} 
\begin{lem} \label{lem:p_base_from_fragment}
\SAVEText
\end{lem}
\newcounter{SAVEthm}
\newcounter{SAVEsection}
\setcounter{SAVEthm}{\value{thm}}
\setcounter{SAVEsection}{\value{section}}
 The proof will be the content of Section~\ref{sec:proof}.

\begin{remark}
While $\cost_{\Omega,R}$ was only defined for infinite~$R$, the definition can be interpreted for finite~$R$, in which case the cost function does not satisfy the limit condition, and the sets obeying it will be the computable ones. \cref{lem:p_base_from_fragment} holds for~$R$ finite or co-finite as well. The case $|R^\complement| < \infty$ tells us nothing: the hypothesis that $X$ fails a $\cost_{\Omega, R^\complement}$-test is trivial, since there is such a test that captures the entire interval; meanwhile, $\cost_{\Omega, R} =^\times \cost_\Omega$, so the conclusion $A \models \cost_{\Omega, R}$ is simply a restatement of the fact that $A$ is $K$-trivial.

The case $|R| < \infty$ is a weaker version of a known result: the assumption that~$X$ fails a $\cost_{\Omega, R^\complement}$-test tells us that $X$ is not $\cost_\Omega$-random, and thus that~$X$ is LR-hard~\cite[Thm.\ 1.5]{Bienvenu.Greenberg.ea:16}. Since~$Y$ is $X$-random, $Y$ is 2-random, and so the only $K$-trivials that~$Y$ computes are the computable sets.	
\end{remark}

\begin{proof}[Proof of \cref{thm:criterion}, assuming Lemma~\ref{lem:p_base_from_fragment}] \NumberQED{thm:criterion}
(i)$\to$(ii) Let $A$ be smart for $\cost_{\Omega, S}$ by \cref{thm:smarts_exist}. Thus $A \models \cost_{\Omega, S}$. Note that $\Omega_S$ is captured by a $\cost_{\Omega, S}$-test by \cref{prop:antichain:capturing}. By \cref{prop:basic fact}, $A \leT \Omega_S$, and so $A \leT \Omega_{R}$. By the definition of smartness for cost functions, $\Omega_{R}$ is captured by a $\cost_{\Omega, S}$-test.

\smallskip
(ii)$\to$(iii) This is the contrapositive of \cref{prop:antichain:non-capturing}.

\smallskip
(iii)$\to$(iv)
Fix $b$ such that $|S\cap m| \le |R\cap m| + b$ for all $m$. Then $|S\cap k(n)| \le |R\cap k(n)| + b$ for all~$n$, meaning that $\limcost_{\Omega, R} \le 2^b\limcost_{\Omega, S}$, and so $\cost_{\Omega,S}\to \cost_{\Omega,R}$.

\smallskip
(iv)$\to$(i) Suppose $\cost_{\Omega,S} \to \cost_{\Omega,R}$, or equivalently $ \limcost_{\Omega,R} \le^\times \limcost_{\Omega,S}$. Since
 \[\limcost_{\Omega,R}\cdot \limcost_{\Omega,R^{{\complement}}} =^\times \limcost_\Omega =^\times
 \limcost_{\Omega,S}\cdot\limcost_{\Omega,S^\complement},\] 
 it follows that $\limcost_{\Omega, S^\complement} \le^\times \limcost_{\Omega, R^\complement}$. By \cref{prop:antichain:capturing}, $X:=\Omega_{S^\complement}$ is captured by a $\cost_{\Omega, S^\complement}$-test, which hence is a $\cost_{\Omega, R^\complement}$-test. By \cref{lem:p_base_from_fragment}, for any $K$-trivial $A \leT \Omega_S:= Y$, $A \models \cost_{\Omega, R}$. Thus $A \leT \Omega_{R}$ by \cref{prop:basic fact}.
\end{proof}

\section{Feeble sets for cost functions, and the structure of ML-degrees}
\noindent We discuss some ramifications of   \cref{thm:criterion}, and also \cref{lem:p_base_from_fragment}, which is a main technical ingredient to the proof of the theorem.

\smallskip

\noindent \emph{A notion that is dual to smartness for a cost function.} 
We will call an ML-random failing a $\cost$-test \emph{feeble for~$\cost$} if the only $K$-trivials it computes are the ones that obey~$\cost$. Recall that a $K$-trivial set obeying a cost function $\cost$ is smart for $\cost$ if only the ML-random sets that fail a $\cost$-test compute it (Def.\ \ref{def:smartness}). Feebleness for $\cost $ is dual to smartness for $\cost$: in each case, by \cref{prop:basic fact} the definition says that the collection of sets of the ``opposite'' type that are Turing comparable to the given set is as small as possible.

%\begin{definition} \label{def:feeble}
%An ML-random sequence~$Y$ is \emph{feeble} for a cost function~$\cost$ if~$Y$ fails a $\cost$-test, and for any $K$-trivial set~$A$, 
%\begin{center}
%$A \models \cost \LR A \leT Y$. 	
%\end{center}
%That is, the only $K$-trivial sets that~$Y$ computes are the ones that it has to compute because they obey~$\cost$.
%\end{definition}
%
As a consequence of \cref{lem:p_base_from_fragment}, we obtain 
a natural characterisation of the $K$-trivial sets that are Turing below $\Omega_R$ for some infinite computable set $R$. 
%It can be seen as a general form of (3)$\leftrightarrow$(4) from \cref{thm:k_n-bases}.

\begin{theorem} \label{thm:Omega_R_is_feeble}
If~$R$ is an infinite computable set, then $\Omega_{R}$ is feeble for $\cost_{\Omega, R}$. (So~the $K$-trivials computable from $\Omega_{R}$ are exactly those that obey $\cost_{\Omega, R}$.)
\end{theorem}
\begin{proof}
By \cref{prop:antichain:capturing}, $\Omega_{R}$ fails a $\cost_{\Omega, R}$-test.

%So by \cref{prop:basic fact}, if $A\models \cost_{\Omega, R}$, then it is computable from $\Omega_R$. 
Suppose $A\leT\Omega_R := Y$  and $A$ is $K$-trivial. The sequence $\Omega_R\oplus\Omega_{R^\complement}$ is ML-random and $X:=\Omega_{R^\complement}$ fails a $\cost_{\Omega, R^\complement}$-test, so $A\models \cost_{\Omega, R}$ by~\cref{lem:p_base_from_fragment}.
\end{proof}

We make two observations that follow from the definitions of smartness and feebleness via \cref{prop:basic fact}. They tell us that cost functions that admit feeble sequences are special. The first observation implies that if~$\cost$ has a feeble random sequence, then the collection of sets that obey~$\cost$ determines a principal ideal of ML-degrees.

\begin{proposition} \label{lem:feeble_characterisation} \

\noindent
Let $\cost$ be a cost function such that $\cost\to \cost_\Omega$. Suppose that~$Z$ is feeble for~$\cost$ and that~$B$ is smart for~$\cost$. Then the following are equivalent for a $K$-trivial set~$A$:
\begin{enumerate}
	\item $A\models \cost$;
	\item $A\leT Z$;
	\item $A\le_\ML B$. 
\end{enumerate}
\end{proposition}

Thus, for example, no random can be feeble for the cost function $\cost_{(A)}$ for the set~$A$ built in the proof of~\cref{thm:shift}.

The next observation says informally that for cost functions admitting feeble sequences, the following are all equivalent: cost function implication, ML-reducibility between corresponding smart sets, and ML$^*$ reducibility between corresponding feeble sets in the sense of Remark~\ref{rem:ML*}.
\begin{proposition} \label{lem:feeble_and_feebleer} %this was misstated in 2017 version
Let $\cost$ and~$\dost$ be cost functions such that $\cost\to \cost_\Omega$ and $\dost\to \cost_\Omega$. Suppose that $Z_\cost$ and~$Z_\dost$ are feeble for~$\cost$ and~$\dost$ (respectively), and that~$B_\cost$ and $B_\dost$ are smart for~$\cost$ and~$\dost$ (respectively). The following are equivalent:
\begin{enumerate}
	\item $\cost\to \dost$;
	\item $B_\cost \le_\ML B_{\dost}$;
	\item Each $K$-trivial Turing below $Z_\cost$ is Turing below $ Z_\dost$, i.e. $Z_\cost \le_{\ML^*} Z_\dost$. 
\end{enumerate}
\end{proposition}
\begin{proof} We obtain from~\cref{lem:feeble_characterisation} that (2) is equivalent to (3):
%]
\begin{center}
$B_\cost \le_\ML B_{\dost}$ iff $\forall \text{$K$-triv.}\, A \, [ A \le_\ML B_\cost \to A \le_\ML B_\dost]$ iff $Z_\cost \le_{\ML^*} Z_\dost$. \end{center}
 
 For (1)$\to$(2), recall  from Section~\ref{s:costf} that (1) is equivalent to $\underline \dost \le^\times \limcost$, and hence implies that every $\dost$-test is a $\cost $-test. If $B_\dost \leT Y$ for ML-random $Y$ then $Y$ fails a $\cost$-test, and hence $B_\cost \leT Y$. 
 
 For (2)$\to$(1), recall that $B_\cost$ is ML-complete for $\cost$. So if $A$ obeys $\cost$ then $A \le_\ML B_\dost$. By~\cref{lem:feeble_characterisation} applied to $\dost$, this implies that $A$ obeys $\dost$.
\end{proof}

%
%\n To verify (1)$\to$(3) 
%
%For (2)$\to$(3)
%For (1)$\to$(2)

\smallskip

\noindent \emph{The structure of the ML-degrees of $K$-trivials.} In~\cref{rem: ML structure}, we provided two facts on the ML-degrees. Both were based on Proposition~\ref{prop:A_is_smart_for_c_A} that each $K$-trivial~$A$ is ML-complete for a cost functions $\cost_{(A)}$. Here we obtain a further structural result using the tools developed above (in this section and the preceding section). In contrast to what the results in \cite{SubclassesPaper} suggested, our result shows that the partial order of ML-degrees is far from linear.

\begin{thm} \label{thm:infinite_antichain} \

\noindent
There is an infinite antichain of ML-degrees of uniformly c.e.\ $K$-trivial sets. Furthermore, every countable partial ordering is embeddable into the ML-degrees of $K$-trivial sets, and via a computable embedding if the given ordering is computable. 
\end{thm} 
\begin{proof} 
We fix a uniformly computable partition of~$\Nat$ into countably many sets~$R_n$ such that the upper density of each~$R_n$ is 1 (greater than $1/2$ would do). For a computable set~$F\subseteq \Nat$, let $R(F) = \bigcup_{n\in F} R_n$. By Theorem~\ref{thm:smarts_exist} let $B_F$ be the c.e.\ $K$-trivial set, uniformly obtained from the cost function, that is smart for~$\cost_{\Omega, R(F)}$. %So for each such~$F$, $B_F$ is computable from~$\Omega_{R(F)}$ and not computable from $\Omega_{{R(F)^\complement}}$. 
The required uniformly c.e.\ antichain is $\seq{B_{\{n\}}}$. The map $F\mapsto B_F$ is an embedding of the partial ordering of computable sets under inclusion into the ML-degrees. This suffices, because one can embed the countable atomless Boolean algebra into the algebra of computable sets under inclusion using its representation as the interval algebra of the countable dense linear order $[0,1)_\mathbb Q$.

To see that the map is a partial order embedding, first suppose that $F\subseteq G$. Then $R(F)\subseteq R(G)$, and so $\cost_{\Omega, R(F)}\ge \cost_{\Omega, R(G)}$; by \cref{lem:feeble_and_feebleer}, $B_F \le_\ML B_G$. 

On the other hand, if $F\nsubseteq G$, take some $n\in F\setminus G$; so $R_n\subseteq R(F)$ but $R_n\cap R(G)=\emptyset$. The fact that the upper density of~$R_n$ is 1 implies that $|R(F)\cap m|\nle^+ |R(G)\cap m|$. 
%By \cref{thm:criterion,thm:Omega_R_is_feeble}, and \cref{lem:feeble_and_feebleer}, $B_F\nle_\ML B_G$. Who writes stuff like this? 
By~\cref{thm:criterion}, this implies that $\cost_{\Omega, R(F)} \not \to \cost_{\Omega, R(G)}$. Hence $B_F \not \le_\ML B_G$ by \cref{thm:Omega_R_is_feeble} and \cref{lem:feeble_and_feebleer}.
\end{proof}

We obtain a related structural result about the ML-degrees without using the tools developed above. We give an alternative construction of incomparable ML-degrees, and use it to prove downward density.

\begin{thm} \label{thm:Ott}
For every non-computable c.e.\ set $D$, there are c.e.\ sets $A, B \leT D$ such that $A \mid_\ML B$. 
\end{thm}
\begin{proof} 
We extend Ku\v cera's injury-free proof~\cite{Kucera:86} of the Friedberg--Muchnik theorem, as presented in \cite[Section 4.2]{Nies:book}. The theorem states that there are Turing incomparable c.e.\ sets $A,B$. Two versions of Ku\v cera's proof are given there; the first relies on $\{0,1\}$-valued d.n.c.\ functions as in \cite[Cor.~4.2.3]{Nies:book}, the second on ML-randomness as in \cite[Cor.~4.2.5]{Nies:book}. The second version actually shows that there are ML-random $\DII$ sets $Y,Z$ such that $A\leT Y$, $B \leT Z$, $A \not \leT Z$, and $B \not \leT Y$. Therefore $A \mid_\ML B$ as witnessed by $Y, Z$. %\andre{Where is the other version?}

To ensure that $A,B \leT D$, all we need to do is modify \cite[Cor.~4.2.5]{Nies:book}: 

\begin{lemma} There is a computable function $r$ such that for each $e$, if $Y= \Phi_e^\Halt$ is total and ML-random, then $A= W_{r(e)} \lwtt Y$, $A \leT D$, and $A$ is non-computable. \end{lemma} 
To see this, we use the cost function version of Kucera's result as presented in~\cite{GreenbergNies:benign} and \cite[5.3.13]{Nies:book}. Given an ML-random $\DII$ set $Y$, one defines a cost function $c_Y$ such that if $A \models c_Y$, then $A \lwtt Y$. The cost function $c_Y$ emulates a given computable approximation of $Y$, and is therefore obtained uniformly from an $e$ such that $Y= \Phi_e^\Halt$. The construction of a non-computable c.e.\ set $A$ obeying a given cost function with the limit condition \cite[Thm.~2.7(i)]{Nies:CalculusOfCostFunctions} is compatible with simple permitting, so we can ensure that $A \leT D$. It is also uniform in the cost function (when $D$ is fixed). So we obtain the c.e.\ set $A$ uniformly in $e$, as required. 
\end{proof}

\begin{remark} The following may be relevant towards Question~\ref{qu:ML arithm} above. We distinguish $\le_\ML$ from certain variants that are clearly arithmetical. Fix a notation $\eta$ for an infinite computable ordinal. Intuitively, a $\DII$ set is $\eta$-c.a.\ if it has a computable approximation where the number of changes is bounded by counting downward in the canonical computable well-order given by $\eta$. See e.g.\ \cite[Def.\ 7.1]{GreenbergHirschfeldtNies:2012} for the formal definition of $\eta$-c.a.\ sets and more background. 

Restricting the ML-randoms in the definition~\ref{def:ML-reducibility} of $\le_\ML$ to the $\eta$-c.a.\ sets yields a reducibility strictly weaker than $\le_\ML$. For, the $\eta$-c.a.\ sets form a $\SI 3$ class, so there is a noncomputable c.e.\ set $D$ below all the $\eta$-c.a.\ ML-randoms. Now by Theorem~\ref{thm:Ott}, let $A, B \leT D$ be c.e.\ sets such that $A \mid_\ML B$. Then $A$ and $B$ are equivalent in the sense of the reducibility based on $\eta$-c.a.\ sets. 
Note that by the proof of Theorem~\ref{thm:Ott}, in fact $A$ and $B$ are incomparable for the weaker variant of ML-reducibility based on ML-random $\DII$ sets. \end{remark} 

%%%%%%%%
%%%%%%%%
\section{Proof of Lemma~\ref{lem:p_base_from_fragment}}
\label{sec:proof}
%%%%%%%%
%%%%%%%%

Recall that for $\+A\subseteq 2^\w$ and $Z\in 2^\w$, the Lebesgue (binary) lower density ${\varrho}_2(\+A|Z)$ of $\+A$ at~$Z$ is $\liminf_n \leb(\+A | Z\rest{n})$, where $\leb(\+A|\s) = \leb(\+A\cap[\s])/\leb([\s])$ is the conditional probability of~$\+A$ given $[\s]$. Notice that ${\varrho}_2(2^\w\times \+ A | X \oplus Z) = {\varrho}_2(\+A|Z)$ for any bit sequence $X$.

 A difference test is one of the form $\seq{\+U_n\cap \+P}$, where the open sets $\+U_n$ are uniformly $\Sigma^0_1$ and nested, $\+P$ is $\Pi^0_1$, and $\leb(\+U_n\cap \+P)\le 2^{-n}$. Franklin and Ng~\cite{FranklinNg} proved that an ML-random sequence $Z$ is difference random (i.e., passes all difference tests) if and only if $Z$ is Turing incomplete.
 
A bit sequence~$Z$ is a \emph{positive density point} if the lower density $\underline{\varrho_2}(\+P|Z)$ is positive for any $\Pi^0_1$ class~$\+P$ that contains~$Z$ (If $Z$ is ML-random, then it makes no difference whether one takes the binary, or the full density defined in the setting of the unit interval.)

%one says that $Z$ is a \emph{density~1 point} if $\underline {\varrho_2}(\+P|Z)=1$ for every $\Pi^0_1$ class containing~$Z$. 
We also require a result from~\cite{BienvenuEtAl:DenjoyDemuthDensity} due to Bienvenu, H\"olzl and two of the authors of the present paper. It implies that an ML-random is difference random if and only if it is a {positive density} point. Furthermore, the failure of these properties will be witnessed on the same~$\Pi^0_1$ classes, an observation we will use below.
\begin{fact}[\cite{BienvenuEtAl:DenjoyDemuthDensity}, Lemma 3.3] \label{lem:density_0_difference_test} Suppose that $\+Q$ is a $\Pi^0_1$-class that contains an ML-random sequence $Z$. Then $Z$ fails a difference test of the form $\seq{\+V_n\cap \+Q}$ iff $\+Q$ has lower density 0 at~$Z$.
\end{fact}
%Recall that a set~$Z$ is \emph{LR-hard} if $Z$-randomness implies ML-randomness relative to~$\emptyset'$, that is, 2-randomness. 
%
%\begin{fact}[\cite{BienvenuEtAl:DenjoyDemuthDensity}, Thm.\ 3.6] \label{lem:density_1_and_LR_hard} \
%
%\noindent
%Any ML-random sequence that is not LR-hard is a density~1 point. 
%\end{fact}

The purpose of this section is to prove \cref{lem:p_base_from_fragment}, which we recall here. 
We will first give a proof in the case that~$A$ is c.e.

\begin{namedthm}{\cref{lem:p_base_from_fragment}}
\SAVEText
\end{namedthm}
\begin{proof}[Proof when~$A$ is c.e.]
\newcounter{TEMPsection}
\newcounter{TEMPthm}
\setcounter{TEMPsection}{\value{section}}
\setcounter{TEMPthm}{\value{thm}}
\setcounter{section}{\value{SAVEsection}}
\setcounter{thm}{\value{SAVEthm}}
Fix a computable enumeration $\seq{A_s}$ of $A$. Fix a $\cost_{\Omega, R^\complement}$-test $\seq{\+U_n}$ that~$X$ fails, and fix a functional~$\Phi$ with $A = \Phi^Y$. Let $\+E$ be the error set for $\Phi$ with respect to~$A$: as before, $\+ E_s$ is the set of oracles~$Z$ such that $\Phi_s^Z$ lies to the left of~$A_s$. Let $\+Q = 2^\w \times (2^\w - \+E)$ (and $\+Q_s = 2^\w \times (2^\w - \+E_s)$).

\smallskip
We carry out a ``ravenous sets'' construction on $\+Q$; for background see \cite[Section 3.1]{SubclassesPaper}. Uniformly in $k, n \in \w$, we enumerate $\Sigma^0_1$ open sets $\+V^k_n\subset 2^\w\times 2^\w$. The \emph{goal} for $\+V^k_n\cap \+Q$ is $2^{-k}(\Omega_{n+1} - \Omega_n)$; we will ensure that no set ever exceeds its goal. In \cite{HirschfeldtNiesStephan:UsingRandomSetsAsOracles} a set playing a role similar to the one of $\+V^k_n$ was called ``hungry'' if it has not reached its goal. The sets $\+V^k_n$ are called ``ravenous'' here, rather than just ``hungry'', because we may feed them with oracle strings that later leave $\+Q$, in which case they get hungry again. 

We will also ensure that $\+V^k_n$ is disjoint from $\+V^k_m$ for $n \neq m$. The parameter~$k$ determines the goal for these ravenous sets; otherwise, the constructions for distinct~$k$ are independent. The other property that we ensure is that
\[ \+V^k_n\cap \+Q\subseteq \+U_n\times \Psi^{-1}[A\rest{n+1}]. \] 

\noindent \emph{Construction of the sets $\+V^k_n$, for parameter $k$.} At stage 0, we begin with $\+V^k_n$ empty for every $ n \in \omega$. At every stage~$s$, we call one of the sets~$\+V^k_n$ ``awake'', and the others ``asleep''. We start with~$\+V^k_0$ awake.

At stage~$s$, if $\+V^k_n$ is awake at this stage, then it has not reached its goal, i.e., $\leb(\+V^k_{n,s}\cap \+Q_s) < 2^{-k} (\Omega_{n+1}-\Omega_n)$, and so we try to feed it. We call a product $[\s]\times [\tau]$ of basic clopen sets {\em palatable} (at stage $s$) if: it is disjoint from $\+V^k_{m, s}$ for all $m$; $[\s]\subseteq \+U_{n,s}$; and $A_s\uhr n+1 \preceq \Phi_s^\tau$ (in particular, $[\s]\times [\tau]$ is covered by $\+Q_s$). Note that if $[\sigma']\times[\tau']$ is contained in a palatable set, it is itself palatable. By standard assumptions on the enumerations of $\Phi$ and $\+U$, we can effectively obtain a finite antichain of palatable sets covering all palatable sets, and so determine the total measure covered by palatable sets. If this measure is less than the appetite of $\+V^k_n$, i.e.\ less than $2^{-k} (\Omega_{n+1}-\Omega_n) - \leb(\+V^k_{n,s}\cap \+Q_s)$, we enumerate all the palatable sets into $\+V^k_{n,s+1}$ and declare that $\+V^k_n$ to be awake at stage $s+1$.

If instead the measure covered by palatable sets at stage $s$ exceeds the appetite of $\+V^k_n$, then we choose a finite anti-chain of palatable sets covering measure exactly $2^{-k} (\Omega_{n+1}-\Omega_n) - \leb(\+V^k_{n,s}\cap \+Q_s)$ in some effective fashion and enumerate this antichain into $\+V^k_{n,s+1}$. We then put~$\+V^k_n$ to sleep %\footnote{Magically, the sets go to sleep when they are told to.} % thought of your kids?
and declare $\+V^k_m$ to be awake at stage~$s+1$, where~$m$ is least such that $\+V^k_{m,s}$ has not reached \emph{half} its goal: $\leb(\+V^k_{m,s}\cap \+Q_s) < 2^{-(k+1)} (\Omega_{n+1}-\Omega_n)$. (Such~$m$ will always exist, of course, because all but finitely many $\+V^k_{m,s}$ will be empty. It is important to note, though, that as we enumerate measure into~$\+V^k_n$, measure leaves $\+Q$, and so a set~$\+V^k_n$ could be put to sleep but re-awakened later.)

\subsubsection*{Verification}

 Since~$X$ fails a $\cost_{\Omega, R^\complement}$ test, it is not 2-random. Since $X$ is ML-random relative to~$Y$, this implies that~$Y$ is Turing incomplete. So by the result of Franklin and Ng, $Y$ is difference random. Since $Y\notin \+E$, by \cref{lem:density_0_difference_test} $2^\w - \+E$ has positive density at~$Y$. Hence $\+ Q$ has positive density at $X \oplus Y$ by the fact mentioned at the beginning of this section. 
 Let $\+V^k = \bigcup_n \+V^k_n$. Since $\sum_n (\Omega_{n+1}-\Omega_n)=\Omega$, the sequence of sets $\seq{\+V^k \cap \+Q}$ is a difference test. By a second application of \cref{lem:density_0_difference_test}, 
 $X\oplus Y \not \in \bigcap_k \+V^k \cap \+Q$.
 
 Since $X\oplus Y \in \+Q$, we can fix some~$k$ with $X \oplus Y \not \in \+V^k_n$ for any $n \in \omega$. In the remainder of the proof, we omit the superscript~$k$. We first show that every set $\+ V_n$ eventually reaches half its goal.

\begin{claim} \label{clm:c.e.case:t}
For every $n$, there is a stage $t$ such that for all $s \ge t$, %$\+V_n$ is asleep at stage $s$, and thus 
\[
\leb(\+V_{n,s} \cap \+Q_s) \ge 2^{-(k+1)}(\Omega_{n+1} - \Omega_n).
\]
\end{claim}
\begin{proof}
\NumberQED{clm:c.e.case:t}
Fix $n$. There is a $\sigma \prec X$ with $[\sigma] \subseteq \+U_n$, and there is a $\tau \prec Y$ with $A\uhr n+1 \preceq \Phi^\tau$. Fix $t_0$ such that $[\sigma] \subseteq \+U_{n,t_0}$, $A_{t_0}\uhr n+1 = A\uhr n+1$ and $A\uhr n+1 \preceq \Phi^\tau_{t_0}$. As each $\+V_{m, s}$ is the union of a finite antichain and does not contain $(X,Y)$, at every $s \ge t_0$ at which $\+V_n$ is awake, there is some palatable $[\sigma'] \times [\tau']$ with $\sigma \preceq \sigma' \prec X$ and $\tau \preceq \tau' \prec Y$. We do not enumerate some neighborhood covering $[\sigma']\times[\tau']$ into $\+V_{n,s+1}$, so by construction, $\+V_n$ is asleep at stage $s+1$.

If $s_0$ is a stage when $\+V_n$ goes to sleep and $s_1 > s_0$ is a stage at which~$\+V_n$ wakes back up, then $\leb(\+Q_{s_0} - \+Q_{s_1}) > 2^{-(k+1)}(\Omega_{n+1} - \Omega_n)$. Thus $\+V_n$ can go to sleep only finitely often. It follows that for every~$n$, there are only finitely many stages at which~$\+V_n$ is awake. Let $t$ be the last stage at which any $\+V_m$ for $m\leq n$ went to sleep. Then $\leb(\+V_{n,s} \cap \+Q_s) \ge 2^{-(k+1)}(\Omega_{n+1}-\Omega_n)$ for every $s \ge t$. For otherwise, when the current $\+V_j$ goes to sleep, either $\+V_n$ or $\+V_m$ for $m<n$ would wake, contrary to the choice of~$t$.
\end{proof}

We now define a pair of computable functions $f$ and $g$ by simultaneous recursion. We begin by setting $f(-1) = -1$. Given $f(s-1)$, we define $f(s)>f(s-1)$ and $g(s)$ to be sufficiently large so that for every $n < s$,
\[
\Omega_{f(s)} - \Omega_{n} \le 2(\Omega_{g(s)} - \Omega_{n}),
\]
and for every $n < g(s)$,
\[
\leb(\+V_{n,f(s)} \cap \+Q_{f(s)}) \ge 2^{-(k+1)}(\Omega_{n+1} - \Omega_n).
\]
Note such values always exist: if $g(s)$ is such that 
$\Omega - \Omega_s \le 2(\Omega_{g(s)} - \Omega_s)$,
then the first requirement is satisfied for every $f(s)$; then given any $g(s)$, a sufficiently large choice of $f(s)$ will satisfy the second requirement. Thus we can find such a pair of values by exhaustive search, and~$f$ and~$g$ are total.

Recall the following notation from Section~\ref{sec:comparability_of_fragments}:
\[
	k_s(n) = \floor{-\log_2 (\Omega_s-\Omega_n)}.
\]
Observe that $k_{f(s)}(n) \ge k_{g(s)}(n) - 1$ for all $n<s$, and so
\[ \cost_{\Omega, R^\complement}(n,f(s)) \le 2\cdot\cost_{\Omega, R^\complement}(n,g(s)). \]
 The following claim will complete the proof that~$A$ obeys $\cost_{\Omega,R}$. 

\begin{claim} \label{clm:c.e.case:b}
The total cost $\cost_{\Omega,R} \seq{A_{f(s+1)}}$ is bounded by $2^{k+3}$.
\end{claim}

\noindent{Proof.} Fix a stage $s$, and suppose that~$n$ be least such that $n \in A_{f(s+1)} - A_{f(s)}$. We may assume $n < s$. Then for all $m \ge n$, $\pi_2[\+V_{m, f(s)}] \subseteq \+E_{f(s+1)}$, where $\pi_2 \colon 2^\w \times 2^\w \to 2^\w$ is the projection onto the second coordinate. Let
\[\+S = \bigcup_{m \ge n} \+V_{m, f(s)} \cap \+Q_{f(s)}.\]
Note that by definition of the $\+ Q_t$
\begin{equation} \label{eq:l1} \leb\left(\+E_{f(s+1)} - \+E_{f(s)}\right) \ge \leb (\pi_2[\+S]). \end{equation}
The sets $\+V_m$ are disjoint by construction, and
\[\leb(\+V_{m, f(s)} \cap \+Q_{f(s)}) \ge 2^{-(k+1)}(\Omega_{m+1} - \Omega_m)\]
for $m < g(s)$ by choice of $f(s)$, so
 %\andre{no restriction on m was assumed so far}
\[\leb(\+S) \ge 2^{-(k+1)}(\Omega_{g(s)} - \Omega_{n}) > 2^{-(k+1)} 2^{-(k_{g(s)}(n)+1)} \ge 2^{-(k+2)}\cost_{\Omega}(n,g(s)). \]
On the other hand, $\pi_1[\+S] \subseteq \+U_{n, f(s)}$, where $\pi_1$ is projection onto the first coordinate, and $\leb (\+U_{n, f(s)}) \le \cost_{\Omega, R^\complement}(n,f(s)) \le 2\cdot \cost_{\Omega, R^\complement}(n, g(s))$. Since $\+S \subseteq \pi_1[\+S] \times \pi_2[\+S]$,
\[	 
\leb(\+S) \le \leb (\pi_1[\+S]) \cdot \leb (\pi_2[\+S]) \le 2\cdot \cost_{\Omega, R^\complement}(n,g(s)) \cdot \leb (\pi_2[\+ S]), 
\]
whence, as $\cost_\Omega \le \cost_{\Omega,R}\cdot \cost_{\Omega,R^\complement}$,
\[
	2^{-(k+3)} \cost_{\Omega, R}(n,g(s)) \le \leb (\pi_2[\+S]).
\]
Therefore, by (\ref{eq:l1}), $\leb(\+E_{f(s+1)} - \+E_{f(s)}) \ge 2^{-k-3}\cdot\cost_{\Omega, R}(n,g(s)) \geq 2^{-k-3}\cdot\cost_{\Omega, R}(n,s)$.

By definition the total cost is the sum over all stages $s$ of the costs of the least change at that stage. We conclude that $\cost_{\Omega,R}\seq{A_{f(s+1)}} \le 2^{k+3}\leb (\+E)$.
\renewcommand\qedsymbol{\ensuremath{\openbox}$_{\textup{Lem}.~\ref{lem:p_base_from_fragment} \text{ for $A$ c.e.}}$}
\renewcommand\qedsymbol{\ensuremath{\openbox}$_{\ref{clm:c.e.case:b},~\textup{Lem}.~\ref{lem:p_base_from_fragment} \text{ for $A$ c.e.}}$}
\end{proof}
\setcounter{section}{\value{TEMPsection}}
\setcounter{thm}{\value{TEMPthm}}

We lift the c.e.\ case to the general case. Thanks to \cref{thm:ML_degrees_ce_generated}, this is relatively easy, say compared to the approach taken in \cite{SubclassesPaper}.

\begin{lemma} \label{lem:c_Omega_R_downward_Turing}
For any infinite computable set~$R$, obedience to $\cost_{\Omega,R}$ is downward closed under Turing reducibility. 
\end{lemma}
\begin{proof}
This is similar to \cite[Prop.~2.3]{SubclassesPaper}. For brevity, let $f(k) = |R\cap k|$. We only use the facts that~$f$ is non-decreasing, and that there is $d\in \NN$ such that $ f(k+1) \le f(k) +d $ for each~$k$. 
As in Section~\ref{sec:comparability_of_fragments} let $k(n) = \floor{-\log_2 (\Omega-\Omega_n)}$ and $k_s(n) = \floor{-\log_2 (\Omega_s-\Omega_n)}$, so that $\lim_s k_s(n)= k(n)$ in a nonincreasing fashion.

Let~$B$ be a $\Delta^0_2$ set that obeys~$\cost_{\Omega,R}$. Let $\seq{B_t}$ be a computable approximation of~$B$ witnessing that $B\models \cost_{\Omega,R}$. Let $A\leT B$, say $A = \Psi^B$ for some functional~$\Psi$. 

 Since $\cost_{\Omega,R}\to \cost_{\Omega}$, $B$ is $K$-trivial. Let~$\psi$ be the use function for the computation $\Psi^B = A$. %\andre{why not $\psi$?}
 By Barmpalias and Downey~\cite[Lem.~2.5]{BD:14}, as~$B$ is $K$-trivial and~$\psi$ is $B$-computable, one has $\Omega - \Omega_n \le^\times \Omega-\Omega_{\psi(n)}$ and hence $k(n)\ge^+ k(\psi(n))$. 
 
 Hence there is~$b\in \mathbb Z$ such that $k(n)\ge b+ k(\psi(n))$ for each $n$. We define an increasing sequence of stages~$s(i)$, starting with $s(0)=0$; $s(i)$ is the least stage $s>s(i-1)$ such that $|\Psi_s^{B_s}|>i$ and for all $n\le i$, $k_{i+1}(n)\ge b+ k_{s}(\psi_s(n))$, where $\psi_s(n)$ is the use of the computation $\Psi_s^{B_s}(n)$. We then let $A_i = \Psi_{s(i)}^{B_{s(i)}}$. 
 
 We claim that the approximation $\seq{A_i}$ witnesses that~$A$ obeys $\cost_{\Omega,R}$. The reason is that if $A_{i}(n)\ne A_{i+1}(n)$ and $n\le i$, then the $A$-cost paid is $2^{-f(k_{i+1}(n))}$, whereas at some stage~$t\in (s(i),s(i+1)]$ we see a change in~$B$ below $v = \psi_{s(i)}(n)$, showing that the total cost paid by~$B$ along this interval of stages is at least $2^{-f(k_t(v))}\ge 2^{-f(k_{s(i)}(v))}$. This allows us to bound the $A$-cost by the assumed property of $f$, as $k_{i+1}(n) \ge b + k_{s(i)}(v)$.
\end{proof}

Note that in fact obedience to $\cost_{\Omega,R}$ is downward closed under ML-reducibility (by \cref{thm:Omega_R_is_feeble}), but this used \cref{lem:p_base_from_fragment}.

\begin{proof}[Proof of Lemma~\ref{lem:p_base_from_fragment} in the general case]
Let~$A$ be $K$-trivial and suppose that the hypotheses of the lemma hold. By \cref{thm:ML_degrees_ce_generated}, let $C\geT A$ be c.e.\ such that $C\equiv_\ML A$. The ML-equivalence implies that $C\leT Y$; the c.e.\ case shows that $C\models \cost_{\Omega,R}$. By \cref{lem:c_Omega_R_downward_Turing}, $A$ obeys $\cost_{\Omega,R}$ as well.
\end{proof}

% We point out a particular consequence of \cref{thm:criterion}. 

% \begin{cor} \label{thm:below_some_columns_below_any} 
% Let $1\le k \le n$. Suppose that $V_0$ and~$V_1$ are joins of~$k$ of the $n$-columns of $\Omega$. If $A$ is $K$-trivial and $A \leT V_0$, then $A \leT V_1$. 
% \end{cor}

%\begin{proof}[Proof of Theorem.] Let $p = k/n$. Let $R$ be such that $V = \Omega_R$. Then $\cost_{\Omega, R}(n,s) =^\times (\Omega_s - \Omega_n)^p$, and $\cost_{\Omega, \ol R}(n,s) =^* (\Omega_s - \Omega_n)^{1-p}$. Let $X = \Omega_{\ol R}$, such that $X \oplus V = \Omega$. By \cref{thm:ML_degrees_ce_generated} together with the fact that obeying $\cost_{R}$ is downward closed under Turing reducibility \cite[xxx]{Greenberg et al}, we may assume that $A$ is c.e.. By \cref{prop:antichain:capturing}, $X$ fails a $\cost_{\Omega, \ol R}$-bounded test. So $A \models \cost_{R}$. Since $W$ fails a $\cost_{R}$-bounded test, this shows $A \leT W $. \end{proof}

%
%
%Version for left-c.e.\ $p$? Not sure this is interesting. Not sure it's true, either: it's important that $p$ be computable, because we want both $\cost_{\Omega, p}$ and $\cost_{\Omega, 1-p}$ to be monotonic. I guess we don't really need $\cost_{\Omega,1-p}$ to be monotonic, but then where would we expect to get a real failing such a test?
%

%\section{Proof of \cref{}}} % (fold)
% 

% section a_proof_of_conj:3.5 (end)

%%%%%%%%
%%%%%%%%
\section{Fragments of~\texorpdfstring{$\Omega$}{Omega} and strong jump-traceability}
\label{sec:SJT}
%%%%%%%%
%%%%%%%%

A cost function~$\cost$ is \emph{benign} \cite{GreenbergNies:benign} if from a rational $\epsilon>0$, we can compute a bound on the length of any sequence $n_1 < s_1 \le n_2 < s_2 \le \cdots \le n_{\ell} < s_{\ell}$ such that $\cost(n_{i},s_{i})\ge \epsilon$ for all $i\le \ell$. For example, $\cost_{\Omega}$ is benign, with the bound being $1/\epsilon$.

By an order function we mean a computable, non-decreasing, and unbounded function. A set~$A$ is \emph{strongly jump-traceable} if for every order function~$h$, for every $\psi$ partial computable in $A$, there is an $h$-bounded c.e.\ trace for~$\psi$; that is, there is a sequence $\seq{T(n)}$ such that $|T(n)|\le h(n)$, $T(n)$ is uniformly c.e., and $\psi(n)\in T(n)$ for all~$n\in\dom\psi$. Figueira et al.\ \cite{FigueiraNiesStephan:SJT} introduced this notion, and built a non-computable c.e.\ of this kind.

Plain (rather than strong) jump traceability is the notion where one existentially quantifies over order functions $h$. For c.e.\ sets this is equivalent to superlowness. Universally quantifying over order functions indeed places a very strong restriction on the computational power of the set. For instance, the strongly jump-traceable sets form an ideal in the Turing degrees \cite{CholakDowneyGreenberg:SJT1,Diamondstone.Greenberg.Turetsky:SJT.Enumerable} which is a proper sub-ideal of the $K$-trivials \cite{DowneyGreenberg:SJT2}. A characterisation that will concern us here is that a set is strongly jump-traceable if and only if it obeys all benign cost functions \cite{GreenbergNies:benign,Diamondstone.Greenberg.Turetsky:SJT.Enumerable}. For more on strong jump-traceability, see the survey article~\cite{Greenberg.Turetsky:SJT.BSL}.

One can characterise strong jump-traceability using computability from ML-random sequences. There is more than one such characterisation. For example, a set is strongly jump-traceable if and only if it is computable from all superlow ML-random sequences~\cite{GreenbergHirschfeldtNies:2012}, also if and only if it is computable from all superhigh random sequences~\cite{GreenbergHirschfeldtNies:2012,Greenberg.Turetsky:SJT.BSL}. Alternatively, a c.e.\ set is strongly jump-traceable if and only if it is computable from a Demuth random sequence by combining~\cite{GreenbergTuretsky:Demuth} and \cite{Kucera.Nies:11}; this extends to all $K$-trivials by \cref{thm:ML_degrees_ce_generated}. As a consequence one has:

\begin{proposition} \label{prop:SJT_ML_ideal}
The strongly jump-traceable sets form an ideal in the ML-degrees. 
\end{proposition}

In \cite[5.3.1]{SubclassesPaper}, it is observed that every strongly jump-traceable set is a $p$-base for all $p>0$. However, this is not a characterisation. The sets that are $p$-bases for all~$p>0$ are the $1/\w$-bases, those which are computable from each column from an infinite partition of some random sequence. Equivalently, they are computable from~$\Omega_R$ for all computable sets~$R$ such that $\liminf_n \, |R \cap n|/n$ is positive. Some such sets are not strongly jump-traceable as pointed out in \cite[5.3.1]{SubclassesPaper}. Here we see that we obtain a characterisation of strong jump-traceability if we drop the density condition. 

\begin{prop}
For any infinite computable set~$R$, $\cost_{\Omega,R}$ is benign.
\end{prop}
\begin{proof}
Given a rational $\epsilon>0$, first, we compute an $m$ with $2^{-|R \cap m|} < \epsilon$. Let $n_1<s_1\le n_2< s_2 < \cdots \le n_\ell< s_\ell$ be a sequence such that for all $i\le \ell$, $\cost_{\Omega, R}(n_i, s_i) > \epsilon$. This means that $k_{s_i}(n_i) < m$, and so $\Omega_{s_i} - \Omega_{n_i} \ge 2^{-m}$. So
\[1 > \Omega
> \sum_{i \le \ell} \Omega_{s_i} - \Omega_{n_i}
\ge \ell\cdot 2^{-m},\]
and thus $\ell< 2^m$.
\end{proof}

\begin{prop} \label{prop:benign_implied_by_c_Omega_R}
For any benign cost function $\cost$, there is an infinite computable set~$R$ with $\cost_{\Omega,R} \to \cost$.
\end{prop}
\begin{proof}
Suppose $g(\epsilon)$ is a computable bound witnessing that $\cost$ is benign. We will construct a left-c.e.\ real $\beta < 1$. Since $\Omega$ is Solovay complete, by the recursion theorem, we may assume that we already know a constant $\delta > 0$ and a computable approximation to $\Omega$ with $\delta(\beta_s - \beta_n) < \Omega_s - \Omega_n$ for all~$n$ and~$s$. Choose a computable sequence $m_0 < m_1 < \cdots$ such that
\[ \sum_i \frac{2^{-m_i}\cdot g\left(2^{-(i+1)}\right)}{\delta} < 1.\]
Let $R = \{ m_0 < m_1 < \cdots\}$.

Define $\beta_0 = 0$. At stage $s+1$, if $\cost(n,s+1) \le \cost_{\Omega,R}(n,s)$ for all~$n$, then let $\beta_{s+1} = \beta_s$. Otherwise, let~$n$ be least with $\cost(n,s+1) > \cost_{\Omega,R}(n,s)$. Let $i = \floor{-\log \cost(n,s+1)}$. Define $\beta_{s+1} = \beta_s + 2^{-m_i}/\delta$. The point of this is to increase~$\Omega$: in this case, we have
\[
	\Omega_{s+1} -\Omega_s > \delta\cdot (\beta_{s+1} - \beta_s) = 2^{-m_i},
\]
and so $k_{s+1}(s)\le m_i$. In turn, this implies that $\cost_{\Omega,R}(s,s+1)\ge 2^{-i}$.

\begin{claim} \label{clm:benignR:1}
For all $n$, $\limcost_{\Omega,R}(n) \ge \limcost(n)$.
\end{claim}
\begin{proof} 
\NumberQED{clm:benignR:1}
We show that $\cost(n,s)\le \cost_{\Omega,R}(n,s)$ for and~$s$ and all~$n<s$. Suppose this holds for~$s$; we verify it for~$s+1$. Suppose there is~$\hat n \le s$, chosen least, such that $\cost(\hat n,s+1) > \cost_{\Omega,R}(\hat n,s)$ (otherwise there is nothing to do for $s+1$). For all $n<\hat n$, 
\[
\cost(n,s+1)\le \cost_{\Omega,R}(n,s)\le \cost_{\Omega,R}(n,s+1).
\] 
Let $i = \floor{-\log \cost(\hat n,s+1)}$. For all~$n\ge \hat n$ such that $n \le s$, 
\[
	\cost(n,s+1) \le \cost(\hat n, s+1) \le 2^{-i} \le \cost_{\Omega,R}(s,s+1) \le \cost_{\Omega,R}(n,s+1),
\]
as required. % \andre{took out the $\infty$ stuff. you need $n \le s$ for monotonicity of $\cost_R$.}
\end{proof}

The proof of the proposition will be complete once we show:

\begin{claim} \label{clm:beta_small}
$\beta < 1$.
\end{claim}

\noindent{\it Proof.} Fix $i$ and let $s_0 < s_1 < s_2 < \cdots$ be the stages $s$ with $\beta_{s+1} - \beta_s = 2^{-m_i}/\delta$. By construction, for every $n \le s_j$, $\cost_{\Omega,R}(n,s_j+1) \ge 2^{-i}$. Also by construction, there is some $n_j \le s_j$ with $2^{-(i+1)} < \cost(n_j, s_{j}+1) \le 2^{-i}$ and $\cost_{\Omega,R}(n_j,s_j) < \cost(n_j,s_j+1)$. Thus $n_0 < s_0+1 \le n_1 < s_1+1 \le \cdots$. It follows that there are at most $g\left(2^{-(i+1)}\right)$ such stages. So
\[\beta = \sum_s \beta_{s+1} - \beta_s \le \sum_i g\left(2^{-(i+1)}\right)\frac{2^{-m_i}}\delta < 1,\] by the choice of $m_i$.
\renewcommand\qedsymbol{\ensuremath{\openbox}$_{\ref{clm:beta_small},~\ref{prop:benign_implied_by_c_Omega_R}}$}	
\end{proof}

It follows that the sets which obey $\cost_{\Omega,R}$ for all computable~$R$ are precisely the strongly jump-traceable sets. \Cref{thm:Omega_R_is_feeble} implies the following, which extends the result from \cite{GreenbergHirschfeldtNies:2012} that a set is strongly jump-traceable if and only if it is computable from all $\w$-computably approximable random sequences. 

\begin{cor} \label{cor:SJT_and_R}
A ($K$-trivial) set~$A$ is strongly jump-traceable if and only if $A\leT \Omega_R$ for every infinite computable set~$R$. 
\end{cor}

\noindent Note that $K$-triviality is for free here, as such a set is a $1/2$-base.

Similarly, we see that a $K$-trivial set~$A$ is strongly jump-traceable if and only if for every infinite computable~$R$, we have $A\le_\ML B_R$, where $B_R$ is ML-complete for $\cost_{\Omega,R}$. That is, the ML-ideal of strongly jump-traceable sets is the intersection of the infinitely many principal ideals given by the sets $B_R$. We conjecture that this ideal is not principal.

\end{document}